\documentclass[12pt]{amsart}
 \usepackage{amsmath,amssymb,amscd,amsthm,esint}
\usepackage{mathrsfs}
 \usepackage{color}
 \usepackage{enumerate}
\usepackage{cite}
\allowdisplaybreaks
\begin{document}

\newtheorem{theorem}{Theorem}[section]
\newtheorem{prop}[theorem]{Proposition}
\newtheorem{lemma}[theorem]{Lemma}
\newtheorem{definition}[theorem]{Definition}
\newtheorem{corollary}[theorem]{Corollary}
\newtheorem{example}[theorem]{Example}
\newtheorem{remark}[theorem]{Remark}
\newcommand{\ra}{\rightarrow}
\renewcommand{\theequation}
{\thesection.\arabic{equation}}

\def\Xint#1{\mathchoice
  {\XXint\displaystyle\textstyle{#1}}%
  {\XXint\textstyle\scriptstyle{#1}}%
  {\XXint\scriptstyle\scriptscriptstyle{#1}}%
  {\XXint\scriptscriptstyle\scriptscriptstyle{#1}}%
  \!\int}
\def\XXint#1#2#3{{\setbox0=\hbox{$#1{#2#3}{\int}$}
  \vcenter{\hbox{$#2#3$}}\kern-.5\wd0}}
\def\ddashint{\Xint=}
\def\dashint{\Xint-}

\def \HS {H^1_{w}(X)}
\def\HSL  {H^{1}_{L,w}(X)}
\def \R  {\mathbb{R}^n}
\def \RN  {\mathbb{R}^{n_1}}
\def \RM  {\mathbb{R}^{n_2}}
\def \HL {H^1_{L,w}(X)}
\def \HHL {\mathbb{H}^{1,q}_{L,at,w}(X)}
\def \LP {L^p_w}
\def\ls{\lesssim}
\def \TentRn {T_{2,w}^p(\mathbb{R}^n)}

\title[A q-atomic decomposition of weighted tent spaces]
{ A q-atomic decomposition of weighted tent spaces on spaces of homogeneous type and its application }

\thanks{{\it {\rm 2010} Mathematics Subject Classification:} Primary: 42B35, 42B30; Secondary: 47F05.}
\thanks{{\it Key words and phrase:} Weighted tent spaces; Atomic decomposition; Weighted Hardy space associated to operators;  Spaces of homogeneous type}
\thanks{* Corresponding author}

\author[L.Song, L.C. Wu]{ Liang Song \ and \ Liangchuan Wu*}

\address{
   Liang Song,
    Department of Mathematics,
    Sun Yat-sen University,
    Guangzhou, 510275,
    P.R.~China}
\email{ songl@mail.sysu.edu.cn}

\address{
   Liangchun Wu,
    Department of Mathematics,
    Sun Yat-sen University,
    Guangzhou, 510275,
    P.R.~China}
\email{wulchuan@mail2.sysu.edu.cn}

\maketitle

\begin{abstract}
 The theory of tent spaces on $\R$ was introduced  by Coifman, Meyer and Stein, including atomic decomposition, duality theory and so on.  Russ generalized the atomic decomposition for tent spaces to the case of  spaces of homogeneous type $(X,d,\mu)$.   The main purpose of  this paper is to extend the results of Coifman, Meyer, Stein and Russ to weighted version.  More precisely,   we obtain a  $q$-atomic decomposition  for  the weighted tent spaces $T^p_{2,w}(X)$, where $0<p\leq 1, 1<q<\infty,$ and $w\in A_\infty$. As an application, we give an atomic decomposition for  weighted Hardy spaces associated to nonnegative self-adjoint operators  on  $X$.
 \end{abstract}

 \medskip

\section {Introduction}
\setcounter{equation}{0}

\bigskip

 Tent spaces on $\R$ were introduced and developed by Coifman, Meyer and Stein in \cite{CMS83,CMS}. These spaces  are very powerful tools in harmonic analysis, when one considers such problems as  atomic decomposition for Hardy space $H^1(\mathbb{R}^n)$, the duality theory between $H^1(\mathbb{R}^n)$ and ${\rm BMO}(\R)$,  Carleson measures. Spaces of homogeneous type were introduced by Coifman and Weiss in \cite{CW}, which are more general settings than $\R$.  The atomic decomposition for tent spaces on spaces of homogeneous type was established  by  Russ in \cite{Ru06}.  The main purpose of  this paper is to give a  $q$-atomic decomposition  for  the weighted tent spaces $T^p_{2,w}(X)$ on  spaces of homogeneous type where $0<p\leq 1, 1<q<\infty,$ and $w\in A_\infty$, which
 will extend the results of Coifman, Meyer, Stein and Russ to weighted tent spaces.

Let us introduce the setting of our paper.   $(X,d,\mu)$ is a metric measure space endowed with a distance $d$ and a nonnegative Borel doubling measure $\mu$ on $X$. Recall that a measure is doubling provided that there exists a constant $C>0$ such that for all $x\in X$ and for all $r>0$,
\begin{equation}\label{eqn:doubling}
	  V(x,2r)\leq C V(x,r)<\infty,
\end{equation}
where $B(x,r)=\{y\in X:\, d(x,y)<r\}$ and $V(x,r)=\mu(B(x,r))$.

Now we recall the definition and atomic decomposition of tent spaces on $X$.

Denote
$$
    \Gamma(x)=\left\{(y,t)\in X\times (0,+\infty):\,  d(x,y)< t\right\}, \ \ {\rm for \ } x\in X.
$$
\begin{definition}\label{def:unweightedTent}
For $p\in (0,\infty)$, the tent space $T_{2}^p(X)$ is defined by
$$
    T_{2}^p (X)=\big\{F(x,t):\  \ F \ {\rm  is \  a\  measurable \  function \ on} \  X\times (0,+\infty), \   {\mathcal A}(F)\in L^p(X)\big\},
$$
equipped with a norm
$$
    \|F\|_{T_{2}^p(X)}:=\left( \int_X \big({\mathcal A}(F)(x)\big)^p \, d\mu(x) \right)^{1/p},
$$
where
$$
    \mathcal{A}(F)(x):=\left(\iint_{\Gamma(x)} |F(y,t)|^2 \,\frac{d\mu(y)}{V(x,t)}\frac{dt}{t}\right)^{1/2}.
$$
\end{definition}

\begin{definition}
Let $p\in (0,1] $.  A measurable function $\mathfrak{a}$  defined in $X\times (0,+\infty)$  is said to be a  2-atom of $T^p_2(X)$ if there exists a ball $B\subseteq X$ such that 

{\rm (i)} \  ${\rm supp }\, \mathfrak{a} \subseteq T(B):=\{(x,t)\in X\times (0,+\infty): \, d(x, X\setminus B)\geq  t\}$;

{\rm (ii)} \  $\displaystyle \left(\iint_{X\times (0,+\infty)}|\mathfrak{a}(y,t)|^2\,\frac{d\mu(y)dt}{t}\right)^{1/2}\leq V(B)^{1/2-1/p}$.
\end{definition}

 Russ \cite{Ru06} obtained the following atomic decomposition theorem for $T_{2}^p(X)$, which generalized the classical result of Coifman, Meyer and Stein \cite{CMS} on $\R$.

\begin{theorem}\label{thm:Russ}
Let $0<p\leq 1$.  For all $f\in T_{2}^p(X)$, there exist a sequence  $\{\lambda_n\}_{n\in \mathbb{N}}\in l^p$ and a sequence of $2$-atoms of  $T_{2}^p(X)$,  $\{a_n\}_{n\in \mathbb{N}}$,  such that $f=\sum_{n=0}^\infty \lambda_na_n$, and
$$
    \sum_{n=0}^\infty|\lambda_n|^p\leq C_p \|f\|_{T_{2}^p(X)}^p.
$$
\end{theorem}

The theory of Muckenhoupt weights introduced  in \cite{Mu}, is  an important  part of harmonic analysis. The main purpose of this article  is to establish a weighted version of  Theorem \ref{thm:Russ}.
We next review some basic facts on Muckenhoupt weights.

A weight $w$  is a non-negative locally integrable function  on $X$ that takes values in $(0,\infty)$ almost everywhere. We say that $w \in A_p(X)$, $1<p<\infty$, if there exists a constant $C$ such that
\begin{equation}\label{def:Ap}
    [w]_{A_p}:=\sup_B \Big( \frac{1}{\mu(B)}\int_B w(x) \,d\mu(x)\Big) \Big(\frac{1}{\mu(B)}\int_B {w(x)}^{-1/(p-1)}\,\,d\mu(x)\Big)^{p-1}\leq C,
\end{equation}
where the supremum is taken over all  balls  $B\subseteq X$.  We say that $w\in A_1$ if there is a constant $C$ such that for every ball $B\subseteq X$,
$$
    \frac{1}{\mu(B)}\int_B w(y)\, d\mu(y) \leq C w(x) \ \ \ \ {\rm for  \ \ a.e.} \ x\in B.
$$
A weight $w$ is in the class $A_\infty$ when it belongs to some $A_p, \ p>1$.

Let $r>1$. We say $w\in {\rm RH}_r$ (the reverse H\"{o}lder classes),  if there is a constant $C$ such that for any ball $B\subseteq X$,
$$
    \Big(\frac{1}{\mu(B)}\int_B w^r(x) \,d\mu(x)\Big)^{1/r}\leq C \frac{1}{\mu(B)}\int_B w(x)\,d\mu(x).
$$
It is known that $f\in A_\infty$ implies that there exists some $r>1$ such that $f\in {\rm RH}_r$. For more properties and related topics about $A_p$ weights, we refer to \cite{St,Gra1,Gra2}.

For $w\in A_\infty$ and $0< p<\infty$, the weighted Lebesgue spaces $L^p_w(X)$ are defined by $\big\{f: \|f\|_{L^p_w(X)}<\infty\big\}$, where $\|f\|_{L^p_w(X)}:=\big(\int_{X} |f(x)|^pw(x)\,d\mu(x)\big)^{1/p}$.

\smallskip

In this article, we are concerned about the following  weighted tent spaces.\footnote{There are also some other weighted tent spaces which are different from that of our paper.  See, for example, \cite{So,Am2}}

\begin{definition}
For $p\in (0,\infty)$ and $w\in A_\infty$, the tent space $T_{2,w}^p(X)$ is defined by
$$
    T_{2,w}^p (X)=\big\{F(x,t):\  \ F \ {\rm  is \   a \  measurable \  function \ on} \  X\times (0,+\infty), \   {\mathcal A}(F)\in L^p_w(X)\big\},
$$
equipped with a norm $\|F\|_{T_{2,w}^p(X)}:=\|{\mathcal A}(F)\|_{L^p_w(X)} .$
\end{definition}

As far as our points are concerned, there are two different types of atoms for weighted tent spaces
in the case of $X=\R$.

\begin{definition}[Weighted atoms of type I]\label{def:tentatom type I}
Let  $p\in (0,1] $ and $w\in A_\infty(\R)$. A measurable function $\mathfrak{a}$  defined in $\mathbb{R}^{n+1}_+$  is called a  $2$-atom of $T^p_{2,w}(\mathbb{R}^n)$ if  there exists a ball $B\subseteq \R$, such that

{\rm (i)}  $\ {\rm supp }\, \mathfrak{a} \subseteq T(B)$;

{\rm (ii)}   $\ \displaystyle \left(\iint |\mathfrak{a}(y,t)|^2w(B(y,t)) \,\frac{dydt}{t}\right)^{1/2}\leq w(B)^{1/2-1/p}$.
\end{definition}

The above definition about weighted atom was introduced by  Harboure, Salinas and Viviani  in \cite{HSV}. Under $w\in A_{1+\frac{1}{n}}(\mathbb{R}^n)$, the authors  obtained  the $2$-atom decomposition characterization for $\TentRn$. (Indeed,  a more general tent space $T_\eta (w)$ was studied in \cite{HSV}, where $\eta$ is  a nonnegative increasing and concave function, and $T_\eta(w)=T_{2,w}^1(\mathbb{R}^n)$ when $\eta(t)=t$.)

Let us see another type of weighted atoms introduced by Bui, Cao, Ky, D. Yang and S. Yang  in \cite{BCKYY}.

\begin{definition}[Weighted atoms of type II]\label{def:tentatom type II}
Let $p\in (0,1] $ and $w\in A_\infty(\R)$. A measurable function $\mathfrak{a}$ on $\mathbb{R}^{n+1}_+$ is called an $\infty$-atom of $T_{2,w}^p(\mathbb{R}^n)$ if there exists a ball $B\subseteq \R$, such that

{\rm (i)} \ ${\rm supp} \  \mathfrak{a} \subseteq T(B)$;

{\rm (ii)} \ for {\bf any} $q\in (1,\infty)$, $\|\mathfrak{a}\|_{T_{2}^q(\mathbb{R}^n)}\leq |B|^{1/q}w(B)^{-1/p}$.
\end{definition}

Under the  weakest  condition on Muckenhoupt weights, $w\in A_\infty(\mathbb{R}^n)$,  Bui {\it et al} (\cite{BCKYY}) showed   an  $\infty$-atomic decomposition  for $T_{2,w}^p(\mathbb{R}^n)$. That is, for all $f\in T_{2,w}^p(\mathbb{R}^n)$, there exist  $\{\lambda_n\}_{n\in \mathbb{N}}\in l^p$ and a sequence of $\infty$-atoms (type II) of  $T_{2,w}^p(\mathbb{R}^n)$,  $\{a_n\}_{n\in \mathbb{N}}$,  such that
$$
    f=\sum_{n=0}^\infty \lambda_na_n \ \quad \ {\rm and} \quad \ \ \|\{\lambda_n\}\|_{l^p}\leq C_p \|f\|_{T_{2,w}^p(\mathbb{R}^n)}.
$$

We note that if $w\in A_\infty$,  it follows  that  any $\infty$-atom (type II) of  $T_{2,w}^p(\mathbb{R}^n)$, $\mathfrak{a}$,  belongs to $T_{2,w}^p(\mathbb{R}^n)$. In fact, by using  H\"older's inequality, one writes

\begin{align*}
\|\mathfrak{a}\|^p_{T_{2,w}^p(\mathbb{R}^n)}&\leq \left[\int_B {\mathcal{A}(\mathfrak{a})(x)}^q \,dx\right]^{p/q} \left[\int_B {w(x)}^{\frac{q}{q-p}} \, dx\right]^{1-\frac{p}{q}} \\
&\leq |B|w(B)^{-1} \left[\frac{1}{|B|}\int_B {w(x)}^{\frac{q}{q-p}}\, dx\right]^{1-\frac{p}{q}}.
\end{align*}
Since $w\in A_\infty$, there exists some $r>1$, such that $w\in {\rm RH}_r$. We can choose $q$ large enough such that $\frac{q}{q-p}\leq r$, then $w\in {\rm RH}_{\frac{q}{q-p}}$, which implies $\|\mathfrak{a}\|^p_{T_{2,w}^p(\mathbb{R}^n)}\leq C$.

It is worth noticing that if we replace (ii) in Definition \ref{def:tentatom type II} by

\quad { $(ii)'$}  \   for {\bf some} $q\in (1,\infty)$, $\|\mathfrak{a}\|_{T_{2}^q(\mathbb{R}^n)}\leq |B|^{1/q}w(B)^{-1/p}$.

\noindent Here, we might as well call it   ``$q$-atom" (type II) of $T_{2,w}^p(\mathbb{R}^n)$.  It is not difficult to see that  $w\in A_\infty$  can not guarantee ``$q$-atom"  (type II) of  $T_{2,w}^p(\mathbb{R}^n)$ belong to $T_{2,w}^p(\mathbb{R}^n)$.

So we may ask one natural question.

\smallskip

\noindent {\bf Question:}  Is it possible to define a suitable $q$-atom of $T_{2,w}^p(X)$ such that $T_{2,w}^p(X)$ admits a $q$-atomic decomposition under $w\in A_\infty(X)$?

\smallskip

In this paper, we give an affirmative answer to the above question.  Let us state our definition of $q$-atom of $T_{2,w}^p(X)$ and the theorem of $q$-atomic decomposition for $T_{2,w}^p(X)$.

\begin{definition}\label{def:tentatom}
Let $p\in (0,1]$, $q\in (1,\infty),$  and $w\in A_\infty(X)$. A measurable function $\mathfrak{a}(x,t)$ on $X\times (0,+\infty)$ is called a $q$-atom of $T_{2,w}^p(X)$, if there exists a ball $B\subseteq X$, such that
	\begin{itemize}
		\item[(i) ] ${\rm supp }\, \mathfrak{a}\subseteq T(B)$;
		
		\medskip
		
		\item[(ii) ]  $\|\mathfrak{a}\|_{T^q_{2,w}(X)}\leq w(B)^{1/q -1/p}.$
	\end{itemize}
\end{definition}

\begin{theorem} \label{thm:tentatom}
Let $p\in (0,1]$, $q\in (1,\infty), $ and $w\in A_\infty(X)$.
\begin{itemize}
\item  [(i) ] Every $F\in T_{2,w}^p(X)$ can be written as $F=\sum_{j=0}^\infty \lambda_j \mathfrak{a}_j$ almost  everywhere, where $\mathfrak{a}_j$ are $q$-atoms of $T_{2,w}^p(X)$, $\lambda_j\in \mathbb{C}$, and
$$
    \Big( \sum_j |\lambda_j|^p\Big)^{1/p}\leq C \|F\|_{T_{2,w}^p(X)},
$$
where $C$ is a positive constant independent of $F$.

\smallskip

\item [(ii) ] Conversely, if $\{\lambda_j\}_{j=0}^\infty \in \ell^p$ and $\mathfrak{a}_j$  are $q$-atoms of $T_{2,w}^p(X)$, then $F=\sum_{j=0}^\infty \lambda_j \mathfrak{a}_j\in T_{2,w}^p(X)$, and
$$
     \|F\|_{T_{2,w}^p(X)}\leq \Big(\sum_{j}|\lambda_j|^p\Big)^{1/p }.
$$
\end{itemize}
\end{theorem}

\begin{remark}
To the best of our knowledge, the $q$-atom of $T_{2,w}^p(X)$  in Definition \ref{def:tentatom} is new even in $X=\R$. We also note that   Definition \ref{def:tentatom type I}  is indeed  our  $q$-atom of $T_{2,w}^1(X)$ for the case of  $q=2$ and  $X=\R$.  So  Definition \ref{def:tentatom} and Theorem \ref{thm:tentatom} can also be viewed as the extension of the corresponding result of \cite{HSV} from $q=2, \ p=1$ and $X=\R$ to $q\neq 2$, $p<1$ and general space of homogeneous type $X$, under $w\in A_\infty(X)$.
\end{remark}

To prove Theorem \ref{thm:tentatom}, we first show a useful dual result of weighted tent spaces (Lemma \ref{lem:tentdual}), that is,  the dual of $T_{2,w}^q(X)$ is $T_{2,w^{-1/(q-1)}}^{q'}(X)$, with $1/q+1/{q'}=1, \  (1<q<\infty)$. With Lemma \ref{lem:tentdual} at our disposal,  we then apply the dyadic structure of space of homogeneous type $X$ and  properties of $A_\infty$ weight,  to establish the $q$-atomic decomposition for $T_{2,w}^p(X)$,  based on the  argument in \cite{CMS,Ru06}.

It is well known that there are close relations between tent spaces and many Hardy spaces (including classical Hardy spaces  and Hardy spaces associated to operators). As an application of Theorem \ref{thm:tentatom}, we will  show an atomic decomposition for  weighted Hardy spaces associated to nonnegative self-adjoint operators  on  spaces of homogeneous type, which generalizes the result \cite{LS} from $\R$ to general space of homogeneous type $X$, via an alternative approach.

Let us review some history about Hardy spaces and weighted Hardy Spaces. The development of the theory of Hardy spaces in ${\mathbb R}^n$ was initiated by Stein and Weiss \cite{SW}. Real variable methods were introduced into this subject
in the seminal article of Fefferman and Stein \cite{FS}.   The atomic decomposition characterization of Hardy spaces was first obtained by Coifman \cite{Co} when $n=1$, and in higher dimensions by Latter \cite{L}. In \cite{G},  Garcia-Cuerva extended  the theory of Hardy space to the weighted case, including atomic decomposition, maximal function characterization and duality (see also \cite{ST}). For more properties about Hardy spaces, we refer to \cite{St,Gra2}.

However, there are some important situations in which the theory of classical Hardy spaces is not applicable. Recently, many researchers have studied Hardy spaces that are  adapted to a linear operator $L$, in much the same way that the classical Hardy spaces are adapted to the Laplacian.  First Auscher, Duong and M$^c$Intosh \cite{ADM}, and then Duong and Yan \cite{DY2}, introduced Hardy   space  adapted to an operator $L$ whose heat kernel enjoys  a pointwise Gaussian upper bound (see also \cite{AR}). In  \cite{AMR} and in \cite{HM}, the authors treated Hardy spaces  adapted, respectively, to the Hodge Laplacian on a Riemannian manifold with doubling measure, or to a second order divergence form elliptic operator on $\mathbb{R}^n$ with complex coefficients, in which settings pointwise heat kernel bounds may fail.  After that, in \cite{HLMMY}, Hofmann {\it et al} considered a non-negative self-adjoint operator $L$ satisfying Davies-Gaffney bounds on $L^2$. They developed a theory of Hardy (and BMO) spaces associated to $L$, including an atomic decomposition and square function characterization.

More recently,  the first-named author  of this paper  and  Yan  presented  a theory of the weighted  Hardy spaces associated to operators \cite{SY}. Roughly speaking, for the non-negative self-adjoint operators $L$ whose heat kernel enjoys  a pointwise Gaussian upper bound, and for $w\in A_1\cap {\rm RH}_2$,  they introduced a weighted Hardy spaces $H^1_{L,w}(\R)$  associated to $L$ in terms of the area function characterization,  and  proved their atomic decomposition characterization. Later, Bui and Duong \cite{AD} extended the result to $H^p_{L,w}$ spaces case (in more general setting) for $0<p\leq 1$. In \cite{LS}, Liu and the first-named author of this paper   obtained  another  atomic decomposition
(the definition of atom is different from  that of \cite{SY}) for $H^1_{L,w}({\mathbb R}^n)$, under the weaker condition on weight, $w\in A_p$ for $1\leq p<\infty$. At the almost same time, for that kind of  atom in \cite{SY}, Bui {\it et al} in \cite{BCKYY} also weakened  the weight conditions in \cite{SY}, and obtain the atomic decomposition.

The following theorem is the second main result of this paper,  which proves the atomic decomposition characterization  for $H_{L,w}^p(X)$, and thus generalizes the result in \cite{LS} from $\R$ to  spaces of homogeneous type $X$.

\begin{theorem}\label{thm:Hardyatom}
Let $w\in A_s, 1<s<\infty$. For $0<p\leq 1$,  $q\geq s$ and $M>\frac{n(s-p)}{2p}$, there holds
$$
    H_{L,w}^p(X)={H}^{p,q,M}_{L,w}(X).
$$
\end{theorem}
\noindent Please see Section 4 for the precise definition of ${H}^{p,q,M}_{L,w}(X)$.

The layout of the article is as follows. In Section 2, we   introduce some  basic properties of spaces of homogeneous type, $A_p$ weights and weighted tent spaces.  In Section 3,  we first show  a dual result for $T_{2,w}^q(X)$, $1<q<\infty$, under the assumption $w\in A_q$  (Lemma \ref{lem:tentdual}). Then based on the  argument in \cite{CMS,Ru06}, we show the first main result of this paper,  Theorem \ref{thm:tentatom}. In Section 4,  we recall  the weighted Hardy spaces $H_{L,w}^p(X)$ associated to a non-negative self-adjoint operator $L$,  and study the  relationship between weighted tent spaces $T_{2,w}^p(X)$ and weighted Hardy spaces $H_{L,w}^p(X)$  (Proposition \ref{prop:TentHardy}). The second main result, Theorem \ref{thm:Hardyatom}, is then obtained  by applying Theorem \ref{thm:tentatom} and Proposition \ref{prop:TentHardy}.

Throughout this article,  we shall denote $w(E):=\int_E w(x)\,d\mu(x)$ for any set $E\subseteq X$.    The letter ``$C$" or ``$c$"  will denote (possibly different) constants that are independent of the essential variables. For $1<q\leq\infty$, we will denote $q'$ the adjoint number of $q$, i.e. ${1}/{q}+ {1}/{q'}=1$.  When $B$ is a ball and $a$ is a positive number, we shall use $aB$ to denote  the $a$-fold dilate of $B$ concentric with $B$.  We also denote $d(E,F)$ the distance of two sets $E$ and $F$ in $X$. The complement of a set $E$ in $X$ will be denoted by $E^c$ or ${^c}E$.

\section{Notation and preliminaries}
\setcounter{equation}{0}

\subsection{Space of homogeneous type and system of dyadic cubes.}

Throughout this article, unless we mention the contrary, $(X,d,\mu)$ is a metric measure space endowed with a distance $d$ and a nonnegative Borel doubling measure $\mu$ on $X$. We also assume $\mu(X)=\infty$ and $\mu(\{x\})=0$ for all $x\in X$.  Note that the doubling property \eqref{eqn:doubling} implies the following strong homogeneity property,
\begin{equation}\label{eqn:doubling2}
	  V(x,\lambda r) \leq C \lambda ^n V(x,r)
\end{equation}
for some $C,\,n>0$ uniformly for all $\lambda \geq 1$ and $x\in X$. The smallest value of the parameter $n$ is a measure of the dimension of the space. There also exist $C$ and $D$ so that
\begin{equation}\label{eqn:doubling3}
	   V(y,r)\leq C \left(1+\frac{d(x,y)}{r}\right)^D  V(x,r)
\end{equation}
uniformly for all $x,\,y\in X$ and $r>0$. Indeed, property \eqref{eqn:doubling3} with $D=n$ is a direct consequence of the triangle inequality for the metric $d$ and the strong homogeneity property \eqref{eqn:doubling2}. In many case such as, for instance, the Euclidean space $\mathbb{R}^n$ or Lie groups of polynomial growth, $D$ can be chosen to be 0.

\smallskip

Next let us introduce some preliminaries on dyadic cubes in $(X,d,\mu)$, which will be used in the dyadic version of Whitney covering theorem. We recall the following construction due to T. Hyt\"onen and A. Kairema~\cite{HK}, which is
a slight elaboration of seminal work by M. Christ \cite{Ch}.

\begin{theorem}\label{thm:dyadic}
Suppose that constants $0<c_0\leq C_0<\infty$ and $\delta\in (0,1)$ satisfy
\begin{equation}\label{eqn:size}
	12C_0\delta\leq c_0.
\end{equation}
Given a set of point $\{x_\beta^k\}_\beta$, $\beta\in \mathscr{J}_k$, for every $k\in \mathbb{Z}$, with the properties that
$$
    d(x_{\beta_1}^k, x_{\beta_2}^k) \  (\beta_1\neq \beta_2)\geq c_o \delta^k,\   \   \     \  \min_\beta d(x,x_\beta^k)<C_0\delta^k    \   \   \   \text{for all  }\ x\in X,
$$	
we can construct a countable family
$$
    \mathscr{D}=\bigcup_{k\in \mathbb{Z}} \mathscr{D}_k,\qquad  \mathscr{D}_k=
    \{Q_\beta^k:\  \beta\in \mathscr{J}_k\},
$$
of Borel sets $Q_\beta^k\subseteq X$, with the properties that
\begin{enumerate}[\rm(i)]
\item $X=\bigcup_{\beta\in \mathscr{J}_k} Q_\beta^k$ (disjoint union) for all $k\in \mathbb{Z}$;

         \smallskip

\item  if $\ell\geq k$, then either $Q_{\beta_2}^\ell\subseteq	 Q_{\beta_1}^k$ or  $Q_{\beta_1}^k\cap Q_{\beta_2}^\ell=\emptyset$;

         \smallskip

\item  for each $(k,\beta_1)$ and each $\ell\leq k$, there exists a unique $\beta_2$ such that $Q_{\beta_1}^k\subseteq Q_{\beta_2}^\ell$;

        \smallskip

\item  there exists a positive constant $M=M(n,\delta)$ such that for each $(k,\beta_1)$,
         \begin{equation}\label{e2.3}
	            1\leq \# \left\{\beta_2:\  Q_{\beta_2}^{k+1}\subseteq Q_{\beta_1}^k\right\}\leq M  \quad  \text{and }\quad 	 Q_{\beta_1}^k= \underset{ \beta_2:\  Q_{\beta_2}^{k+1}\subseteq
               Q_{\beta_1}^k }    {\bigcup} Q_{\beta_2}^{k+1} \ ;
        \end{equation}

        \smallskip

\item  for each $(k,\beta)$,
         \begin{equation}\label{eqn:dycube1}
                B(x_\beta^k, c_0\delta^k/3)\subseteq Q_\beta^k\subseteq B(x_\beta^k, 2C_0\delta^k)=:B(Q_\beta^k),
         \end{equation}
         for some $x_\beta^k\in Q_\beta^k$;

         \smallskip

\item  if $\ell\geq k$ and $Q_{\beta_2}^\ell\subseteq Q_{\beta_1}^k$, then $B(Q_{\beta_2}^\ell)\subseteq B(Q_{\beta_1}^k)$.
\end{enumerate}
\end{theorem}

 $\mathscr{D}$ is called  a system of dyadic cubes, and the the set $Q_\beta^k\in \mathscr{D}_k$ is called a dyadic cube of generation $k$ with center point $x_\beta^k$.   Notably, one could construct systems of dyadic cubes in more general settings such as geometrically doubling  quasi-metric spaces, see \cite{HK}.

We introduce the Whitney decomposition in $(X,d,\mu)$. It's well known that  one can establish a dyadic (cube) version of the Whitney covering theorem in Euclidean spaces. Moreover, in a metric space and given any open set $O$ with $O^c\neq\emptyset$, there exists a set of finite overlapping  balls $\mathscr{B}=\{B_\alpha\}_{\alpha\in \Lambda}$ and a positive constant $c_1$ independent of $O$ such that $O=\bigcup_{\alpha\in \Lambda}  B_\alpha$ and $c_1B_\alpha \cap O^c\neq \emptyset$. See \cite[Theorem~1.3, Chapter III]{CW}  for details. Due to Theorem~\ref{thm:dyadic}, i.e., the existence of system of dyadic cubes in $(X,d,\mu)$, we have the following dyadic Whitney decomposition.

\begin{lemma}\label{lem:Whitney}
	Let $\Omega\subseteq X$ be a proper open subset. Then there exists a set of countable dyadic cubes $\{Q_\alpha\}_{\alpha}$ such that
\begin{itemize}
  \item[(i) ]  The cubes $\{Q_\alpha\}$ are mutually disjoint;
  \item[(ii) ] $\Omega=\bigcup_{\alpha}  Q_\alpha$;
  \item[(iii) ]  ${\rm diam }(Q_\alpha)\leq d(Q_\alpha,\Omega^c)\leq \delta^{-2} {\rm diam }(Q_\alpha)$, where $\delta$ is the parameter in Theorem~\ref{thm:dyadic}.
\end{itemize}
\end{lemma}

\begin{proof}
The proof is standard. For the reader's convenience,  we give a sketch here. Recall that $  \mathscr{D}=\bigcup_{k\in \mathbb{Z}} \mathscr{D}_k$,  $\mathscr{D}_k= \{Q_\beta^k:\  \beta\in \mathscr{J}_k\},$ is a system of dyadic cubes in $X$. For any $k\in \mathbb{Z}$, denote
$$
    \Omega_k:=\left\{x\in \Omega: \, 8C_0\delta^k\leq d(x,\Omega^c)< 8C_0\delta^{k-1}\right\},\   \   \    \   \mathcal{F}_k:=\left\{Q\in \mathscr{D}_k:\,   Q\cap \Omega_k\neq \emptyset \right\},
$$
where $C_0$ and $\delta$ are parameters in Theorem~\ref{thm:dyadic}. It's clear that $\Omega=\bigcup_{k\in \mathbb{Z} }\Omega_k$.

Let $\mathcal{\widetilde{F}}=\bigcup_{k\in \mathbb{Z}}\mathcal{F}_k$ , one may verify that
$$
    \Omega=\bigcup_{Q\in \mathcal{\widetilde{F}}} Q.
$$
By (ii) of Theorem~\ref{thm:dyadic}, for any two distinct dyadic cubes $Q'$ and $Q''$ of $\mathscr{D}$, either $Q'\cap Q''=\emptyset$ or one cube is contained in another one. Therefore,
$$
     \Omega=\bigcup_{Q\in \mathcal{F}} Q,   \   \    \   \text{where}\   \    \mathcal{F}:=\left\{Q\in \mathcal{\widetilde{F}}:\    Q \text{   is maximal}\right\},
$$
and it's clear that all cubes in $\mathcal{F}$ are mutually disjoint.

It remains to show {\rm (iii)}. Note that for any $Q\in \mathcal{F}$, then there exists unique $k\in \mathbb{Z}$ such that $Q\in \mathcal{F}_k$. Then one may apply \eqref{eqn:dycube1} and the definition of $\mathcal{F}_k$ to see
$$
    d(Q,\Omega^c)\geq d(\Omega_k,\Omega^c)-{\rm diam }(Q)\geq 8C_0\delta^k -4C_0\delta^k\geq  {\rm diam }(Q),
$$
and
$$
     d(Q,\Omega^c)\leq \sup_{x\in \Omega_k} d(x,\Omega^c)=8C_0\delta^{k-1}=12C_0c_0^{-1}\delta^{-1} (2c_0\delta^k/3)\leq \delta^{-2}   {\rm diam }(Q),
$$
where the last inequality have used the size condition \eqref{eqn:size}. Thus we finish the proof of Lemma~\ref{lem:Whitney}.
\end{proof}

\smallskip

\subsection{Muckenhoupt weights}

\smallskip
We sum up some  properties of $A_p(X)$ classes.

\begin{lemma}\label{lee2.1}
We have the following properties:

\begin{enumerate}[\rm (i)]
	\item $A_1(X)\subseteq A_p(X) \subseteq A_q(X)$ for $1\leq p\leq q \leq \infty$.
	
	\smallskip
	
	\item If $w\in A_p(X)$, $1<p\leq \infty$, then there exists $1<q<p$ such that $w\in A_q(X)$.

    \smallskip

   \item Let $1<p<\infty$.  Then $w\in A_p(X)$ if and only if $w^{1-p'}\in A_{p'}(X).$

   \smallskip

   \item If $w\in A_p(X)$, $1\leq p<\infty$, then for any ball $B\subseteq X$ and   $E\subseteq B$, there holds
   $$
    \frac{w(B)}{w(E)}\leq [w]_{A_p} \left(\frac{\mu(B)}{\mu(E)}\right)^p.
   $$
\end{enumerate}
\end{lemma}

\begin{proof}
The proof is standard,  and we refer to  \cite{ST} for instance. 	
\end{proof}

\subsection{Denotations and properties of weighted tent space}
For a closed subset $F$ of $X$, let $\mathcal{R}_\alpha (F)$ be the union of all cones with vertices in $F$,
$$
    \mathcal{R}_\alpha (F) =\bigcup_{x\in F} \Gamma_\alpha (x),
$$
where $\Gamma_\alpha(x)=\left\{(y,t)\in X\times (0,+\infty):\,  d(x,y)< \alpha t\right\}$.
Denote $O=F^c$ and assume that $\mu(O)<+\infty$, the tent over $O$ with aperture $\alpha$ is defined by
$$
      T_\alpha(O)=(\mathcal{R}_\alpha(O^c))^c=\{(x,t)\in X\times (0,+\infty): \, d(x,O^c)\geq \alpha t\}.
$$
In the sequel, we write $\Gamma(x)$, $\mathcal{R}(F)$ and $T(O)$ instead of $\Gamma_1(x)$, $\mathcal{R}_1(F)$ and $T_1(O)$, respectively.

For any fixed $\gamma\in (0,1)$, say that $x\in X$ has global $\gamma-$density with respect to $F$ if
$$
        \frac{\mu(B\cap F)}{\mu(B)}\geq \gamma
$$
for all ball $B$ centered at $x$. The set of all such $x$'s is denoted by $F^*$. It's a closed subset of $F$. Define also $O^*=(F^*)^c$. It's clear that $O\subseteq O^*$. Moreover,
\begin{equation}\label{eqn:O*}
	  O^*=\big\{x:\, \mathcal{M}(\chi_O)(x)>1-\gamma\big\},
\end{equation}
where $\mathcal{M}$ denotes the  Hardy-Littlewood maximal function on $X$ and $\chi_O$ is the characteristic function on $O$. As a consequence,
\begin{equation}\label{eqn:O*2}
    \mu(O^*)\leq C_\gamma \mu(O).
\end{equation}

\smallskip

The following key estimate will be used in the sequel,  which was  established by Coifman, Meyer and Stein in \cite{CMS} (for $X=\R$) and by  Russ in \cite{Ru06} (for spaces of homogeneous type $X$).

\begin{lemma}\label{tent space}
Let $\eta\in (0,1)$. Then there exist $\gamma\in (0,1)$ and $C_{\gamma, \eta,n}>0$, where $n$ is the parameter in \eqref{eqn:doubling2}, such that, for any closed subset $F$ of $X$ with $\mu(F^c)<+\infty$ and any nonnegative measurable function  $H(y,t)$ on $X\times (0,+\infty)$.   Then
\begin{equation}\label{eqn:tent}
	\iint_{\mathcal{R}_{1-\eta}(F^*)} H(y,t)V(y,t)\,d\mu(y)dt \leq C_{\gamma,\eta,n} \int_F \left(\iint_{\Gamma(x)} H(y,t) \,d\mu(y)dt\right) d\mu(x),
\end{equation}
where $F^*$ denotes the set of points in $X$ with global density $\gamma$ with respect to $F$.
\end{lemma}

One important observation for weighted tent spaces  will be used  in the paper.
\begin{lemma}\label{lem:L2 with compact} For any compact subset $K$ of $X\times (0,+\infty)$, we have
$$
   \|F\|_{L^2(K,d\mu(x)\frac{dt}{t})}\leq C(K,w,p) \|F\|_{T_{2,w}^q(X)}
$$
for every $F\in T_{2,w}^q(X)$ with ${\rm supp}\, F\subseteq K$, $1<q<\infty$ and $w\in A_q$.
\end{lemma}
\begin{proof}
 This can be deduced by the unweighted estimate  $\|F\|_{L^2(K,d\mu(x)\frac{dt}{t})}\leq C(K,w,p) \|F\|_{T_{2}^1(X)}$ (see  (1.3) in\cite{CMS} or \cite[Lemma 3.3]{Am}) and $\|F\|_{T_{2}^1(X)}\leq  (w^{-1/(q-1)}(\widetilde{K}))^{1/{q'}}  \|F\|_{T_{2,w}^q(X)} $, where $\widetilde{K}={\rm supp}\, \mathcal{A}(F)$.
\end{proof}

\bigskip

\section{Duality and atomic decomposition for weighted tent spaces}
\setcounter{equation}{0}

\subsection{Duality  of  weighted tent spaces $T_{2,w}^q(X)$}

 We first  prove the duality result for weighted tent spaces  $T_{2,w}^q(X)$ for $1<q<\infty$, which will play an important role in the proof of Theorem \ref{thm:tentatom}.

\begin{lemma}\label{lem:tentdual}
Let $1< q<\infty$ and $w\in A_q(X)$. Then the dual of $T_{2,w}^q(X)$ is $T_{2,w^{-1/(q-1)}}^{q'}(X)$, with $1/q+1/{q'}=1$. More precisely, the pairing
$$
    \langle F, G\rangle :=\iint_{X\times (0,+\infty)}  F(x,t) G(x,t) \,d\mu(x)\frac{dt}{t},
$$
realizes $T_{2,w^{-1/(q-1)}}^{q'}(X)$ as equivalent with the dual of $T_{2,w}^q(X)$.
\end{lemma}

\begin{proof}
The proof is inspired by that of \cite [Theorem~2]{CMS} and \cite[Proposion~3.10]{Am}, while  some major modifications are needed due to the presence of the weight $w$.

For every  $F\in T_{2,w}^q(X)$ and $G\in T_{2,w^{-1/(q-1)}}^{q'}(X)$, where  $w\in A_\infty$ (in this part, the condition of $w$ can be weaken to $w\in A_\infty$ rather than $w\in A_q$), we  apply the doubling volume property, Fubini's theorem and H\"older's inequality to compute
	\begin{eqnarray*}
		& &\left|\iint_{X\times (0,+\infty)}  F(y,t) G(y,t) \,d\mu(y)\frac{dt}{t}\right| \\
		&\leq &   \iint_{X\times (0,+\infty)}  |F(y,t)|  |G(y,t)| \int_{B(y,t)} \frac{d\mu(x)}{V(y,t)}
		\,d\mu(y)\frac{dt}{t}\\
		&\leq& C_n \int_{X} \iint_{ \Gamma(x)} |F(y,t)||G(y,t)| \,\frac{d\mu(y)}{V(x,t)} \frac{dt}{t} \,d\mu(x)\\
		&\leq& C_n \int_X \mathcal{A}(F)(x) \mathcal{A}(G)(x) \,d\mu(x)\\
		&\leq& C_n \|\mathcal{A}(F)\|_{L_w^q(X)} \|\mathcal{A}(G)\|_{L_{w^{-1/(q-1)}}^{q'}(X)}\\
		&=& C_n \|F\|_{T_{2,w}^q(X)} \|G\|_{T_{2,w^{-1/(q-1)}}^{q'}(X)}.
		\end{eqnarray*}
As a result, we have showed that for $1<q<\infty$ and $w\in A_\infty$,
every $G\in T_{2,w^{-1/(q-1)}}^{q'} (X)$ induces a  bounded linear functional on $T_{2,w}^q(X)$, via $F \mapsto \iint_{X\times (0,+\infty)}  F(x,t) G(x,t) \,d\mu(x)\frac{dt}{t} $.

\smallskip

Let us prove the converse direction. Assume $w\in A_q$.

\smallskip

\noindent{\it Case 1:} ~ $1<q\leq 2$.

Suppose that $\ell(\cdot)$ is a bounded linear functional on $T_{2,w}^q(X)$ and we denote its norm by $\|\ell\|$. Notice that whenever $K$ is a compact set in $X\times (0,+\infty)$, and ${\rm supp}\, F\subseteq K$ with $F\in L^2(K, d\mu(x)\frac{dt}{t})$, we have ${\rm supp}\, \mathcal{A}(F)=:\widetilde{K}$,  which is a compact set  in $X$, determined by $K$, and
\begin{eqnarray*}
	\|F\|_{T_{2,w}^q(X)}&=&\left(\int_{X}  \big(\mathcal{A}(F)(x)\big)^q w(x) \,d\mu(x)\right)^{1/q}\\
	&\leq&  \left(\int_{\widetilde{K}}  \big(\mathcal{A}(F)(x)\big)^2 w(x) \,d\mu(x)\right)^{1/2} \left(\int_{\widetilde{K}} w(x)\,d\mu(x)\right)^{1/q-1/2}\\
	&=& \left(\iint_K |F(y,t)|^2 \int_{B(y,t)} \frac{w(x)\,d\mu(x)}{V(x,t)} \,d\mu(y)\frac{dt}{t}\right)^{1/2} w(\widetilde{K})^{1/q-1/2}\\
	&\leq & C_n \left(\iint_K |F(y,t)|^2 \frac{w(B(y,t))}{V(y,t)}\, d\mu(y)\frac{dt}{t}\right)^{1/2}w(\widetilde{K})^{1/q-1/2}.
\end{eqnarray*}
Notice that  there exist a  constant $C>0$ and some ball $B_{\max}\subseteq X$ such that for every  $(y,t)\in K$,
$$
    B(y,t)\subseteq B_{\max}\subseteq B(y,Ct),
$$
since $K\subseteq X\times (0,+\infty)$ is compact and then $\min_{(y,t)\in K}t>0$. Hence
$$
    \frac{w(B(y,t))}{V(y,t)}\leq C_n  \frac{w(B_{\max})}{V(B_{\max})}<\infty
$$
for any $(y,t)\in K$.  Therefore, for ${\rm supp}\, F\subseteq K $ with $F\in L^2(K,d\mu(x)\frac{dt}{t})$, we obtain
 $$
      \|F\|_{T_{2,w}^q(X)}\leq C(n,q,w,K) \|F\|_{L^2(K,d\mu(x)\frac{dt}{t})}
 $$
 for $1<q\leq 2$ and $w\in A_\infty$, and so
 $$
     |\ell(F)|\leq \|\ell\|\, \|F\|_{T_{2,w}^q(X)}\leq  C\|\ell\| \|F\|_{L^2(K,d\mu(x)\frac{dt}{t})},
 $$
 which implies that  $\ell(\cdot)$ induces a bounded linear functional on $L^2(K,d\mu(x)\frac{dt}{t})$, and is thus representable by a function $G=G_K\in L^2(K,d\mu(x)\frac{dt}{t})$ from the Riesz representation theorem (see \cite{Y} for example).  Furthermore, for any compact $K_1,K_2\subseteq X\times (0,+\infty)$
with $K_1\subseteq K_2$,  we can verify that  $G_{K_2}(x,t)=G_{K_1}(x,t)$ for  $(x,t)\in K_1$ by contradiction  and  testing $F=\chi_{K_1}\overline{(G_{K_2} -G_{K_1})}$.
 Taking an increasing family of such $K$ which exhaust $X\times (0,+\infty)$, gives us a function $G$, which  is locally in $L^2(X\times (0,+\infty), d\mu(x)\frac{dt}{t})$ and $G\chi_K=G_K$, and so that $\ell(F)=\iint_{X\times (0,+\infty)} F(x,t)G(x,t) \,d\mu(x)\frac{dt}{t}$, whenever $F\in T_{2,w}^q(X)$ with compact support. (Such $F$ are in  $L^2(X\times (0,+\infty), d\mu(x)\frac{dt}{t})$ with compact support by Lemma \ref{lem:L2 with compact}). By noting that the set of such $F$ is dense in $T_{2,w}^q(X)$,  it remains to show that for any compact set $K$ in $X\times (0,+\infty)$,
 \begin{equation}\label{eqn:tentdual1}
 	 \|G_K\|_{T_{2,w^{-1/(q-1)}}^{q'}(X)} \leq C\|\ell\|
 \end{equation}
where the constant $C$ is independent of the compact set $K$.

\smallskip

Next, we will prove \eqref{eqn:tentdual1} in  two subcases as follows.

\noindent{\it Subcase 1-1:} ~ $q=2$.

 We apply the doubling volume property and Fubini's theorem again to see
 \begin{eqnarray*}
    \|G_K\|_{T_{2,w^{-1}}^2(X)}^2 &=& \int_X |\mathcal{A}(G_K)(x)|^2 w^{-1}(x)\,d\mu(x)\\
    &\leq&\int_X \iint_{\Gamma(x)} |G_k(y,t)|^2 \,\frac{d\mu(y)}{V(y,t)}\frac{dt}{t}w^{-1}(x) \,d\mu(x)\\
    &\leq& C_n \iint_{X\times (0,+\infty)}|G_K(y,t)|^2 M_t(w^{-1})(y) \,d\mu(y)\frac{dt}{t},
 \end{eqnarray*}
where
 $$
      M_t(w^{-1})(y):=\frac{1}{V(y,t)} \int_{B(y,t)}  w^{-1}(x)\,d\mu(x).
 $$
 Note that the last integral above  can be written as $\langle \tilde{F}, G_K\rangle$ with $\tilde{F}(y,t)=\overline{G_K(y,t)} M_t(w^{-1})(y)$.
 Hence,
 $$
    \|G_K\|_{T_{2,w^{-1}}^2(X)}^2\leq C_n \, \ell(\tilde{F})\leq C_n \|\ell\| \|\tilde{F}\|_{T_{2,w}^2(X)}.
 $$
 In order to prove \eqref{eqn:tentdual1}, it suffices to show there exists  $C>0$, independent of $K$, such that
\begin{align}\label{e3.4}
    \|\tilde{F}\|_{T_{2,w}^2(X)}\leq C \|G_K\|_{T_{2,w^{-1}}^2(X)}.
\end{align}
By noting that, $M_t (w^{-1})(y) M_t w(y) \leq [w]_{A_2}$ for any $y\in X, t>0$,  we then obtain
\begin{eqnarray*}
    \|\tilde{F}\|_{T_{2,w}^2(X)}^2 &\leq& C_n \iint_{X\times (0,+\infty)} |G_K(y,t)|^2 |M_t(w^{-1})(y)|^2 M_tw(y)\, d\mu(y)\frac{dt}{t}\\
    &\leq& C_{n,w}  \iint_{X\times (0,+\infty)} |G_K(y,t)|^2 M_t w^{-1}(y)\, d\mu(y)\frac{dt}{t}\\
    &\leq& C_{n,w} \int_X w^{-1}(x) \iint_{\Gamma(x)} |G_K(y,t)|^2\, \frac{d\mu(y)}{V(x,t)}\frac{dt}{t} \,d\mu(x)\\
    &=&C_{n,w} \|G_K\|_{T_{2,w^{-1}}^2(X)}^2,
\end{eqnarray*}
which yields  \eqref{e3.4}.

\smallskip

\noindent{\it Subcase 1-2:} ~ $1<q<2$.

Denote
$$
    \Omega_j=\big\{x\in X:\,  2^j\leq \mathcal{A}(G_K)(x)<2^{j+1}\big\},   \ \ {\rm for \  any} \ j\in \mathbb{Z}.
$$
Then
\begin{eqnarray*}
   & &\|G_K\|_{T_{2,w^{-1/(q-1)}}^{q'}(X)}=\|\mathcal{A}(G_K)w^{-1/q}\|_{L^{q'}(X)}\\
   &=&\sup_{\|\psi\|_{L^q(X)} \leq 1} \int_X \mathcal{A}(G_K)(x)w^{-1/q}(x) \psi(x)\, d\mu(x)\\
   &\leq &   \sup_{\|\psi\|_{L^q(X)} \leq 1}   \sum_{j\in \mathbb{Z}} 2^{-j} \int_{\Omega_j} \left( \mathcal{A}(G_K)(x) \right)^2 w^{-1/q}(x) \psi(x)\, d\mu(x)\\
   &\leq & C_n  \sup_{\|\psi\|_{L^q(X)} \leq 1}   \iint_{X\times (0,+\infty)} |G_K(y,t)|^2 M_t\Big(\sum_{j\in \mathbb{Z}} 2^{-j}\chi_{\Omega_j} \psi w^{-1/q}\Big)(y) \, d\mu(y)\frac{dt}{t}\\
   &=& C_n  \sup_{\|\psi\|_{L^q(X)} \leq 1}    \langle  H , G_K  \rangle
   \leq   C_n  \sup_{\|\psi\|_{L^q(X)} \leq 1} \|\ell\| \|H\|_{T_{2,w}^q(X)},
\end{eqnarray*}
where $H(y,t)=\overline{G_K}(y,t) M_t\left(\sum_{j\in \mathbb{Z}} 2^{-j}\chi_{\Omega_j} \psi w^{-1/q}\right)(y)$. Hence it suffices to show
$$
     \|H\|_{T_{2,w}^q(X)}\leq C.
$$
To this end, notice that
$$
    \mathcal{A}(H)(x)\leq \mathcal{A}(G_K)(x) \mathcal{M}\left(\sum_{j\in \mathbb{Z}} 2^{-j}\chi_{\Omega_j} \psi w^{-1/q}\right)(x),
$$
where  $\mathcal{M}$ denotes the Hardy-Littlewood maximal operator. Notice that  $[w\chi_{\Omega_j}]_{A_q}\leq [w]_{A_q}$ uniformly for every $j\in \mathbb{Z}$. Hence we may apply the boundedness of Hardy-Littlewood maximal operators on $L^q_{w\chi_{\Omega_i}}$ (see \cite[Theorem~1.3]{HPR}) to obtain
\begin{eqnarray*}
      \|H\|_{T_{2,w}^q(X)}^q &\leq &   \Big\|\mathcal{A}(G_K) \mathcal{M}\big(\sum_{j\in \mathbb{Z}} 2^{-j}\chi_{\Omega_j} \psi w^{-1/q}\big)\Big\|_{L_w^q(X)}^q\\
      &\leq& C \sum_{i\in\mathbb{Z}} \int_{\Omega_i} 2^{iq}  \mathcal{M}^q\big(\sum_{j\in \mathbb{Z}} 2^{-j}\chi_{\Omega_j} \psi w^{-1/q}\big)(x) \big(w \chi_{\Omega_i} \big)(x) \,d\mu(x)\\
      &\leq& C\sum_{i\in\mathbb{Z}}  2^{iq} \left([w\chi_{\Omega_i}]_{A_q}\right)^{\frac{q}{q-1}}  \int_{X} \Big| \sum_{j\in \mathbb{Z}} 2^{-j}\chi_{\Omega_j} \psi w^{-1/q}\Big|^q (x)  \big(w \chi_{\Omega_i} \big)(x) \,d\mu(x)\\
      &\leq& C [w]_{A_q}^{\frac{q}{q-1}}\sum_{i\in\mathbb{Z}}  2^{iq} \int_{X} \Big| \sum_{j\in \mathbb{Z}} 2^{-j}\chi_{\Omega_j} \psi w^{-1/q}\Big|^q (x)  \big(w \chi_{\Omega_i} \big)(x) \,d\mu(x)\\
      &\leq & C_{q,w} \sum_{i\in\mathbb{Z}} \int_{\Omega_i } |\psi(x)|^q \,d\mu(x)\leq C_{q,w},
\end{eqnarray*}
as desired. So we finish the proof of Lemma~\ref{lem:tentdual} for $1<q<2$.

\medskip

\noindent{\it Case 2:} $2<q<\infty$.

To show the Lemma~\ref{lem:tentdual} for $2<q<\infty$, it suffices to prove that the space $T_{2,w}^q(X)$ is reflexive  for $w\in A_q$ and $1<q\leq 2$.  The remaining argument is to make use of the Eberlein-Smulyan theorem and Lemma \ref{lem:L2 with compact}, that is very  similar to the last part for the proof of Theorem~2 in \cite{CMS}. We skip the details.
\end{proof}

\smallskip

\begin{remark}
We should note that in the case of $X=\R$,  the dual of   $T^1_{2,w}(\mathbb{R}^n)$ is $T^\infty_{2,w}(\mathbb{R}^n)$ if $w\in A_\infty(\mathbb{R}^n)$. It has been proved by Cao, Chang, Fu and Yang in \cite{CCFY}. Their proof can be extend to general spaces of homogeneous type with minor modification. Here, we do not treat this endpoint result since it is uninvolved  in the proof of Theorem \ref{thm:tentatom}.
\end{remark}

\medskip

\subsection{Proof of Theorem~\ref{thm:tentatom}.}
Let us first prove  (ii).	Suppose that $F=\sum_{j=0}^\infty \lambda_j \mathfrak{a}_j(x,t)$, where $\mathfrak{a}_j(x,t)$ are $q$-atoms of $T_{2,w}^p(X)$ and $\{\lambda_j\}_{j=0}^\infty \in \ell^p$.
It follows from Definition~\ref{def:tentatom} that for  every $\mathfrak{a}_j$, there exists a ball $B_j\subseteq X$ such that
$$
     {\rm supp}\, \mathfrak{a}_j\subseteq T(B_j)   \ \ {\rm and}\ \ 	    \left(\int_X  \big(\mathcal{A}(\mathfrak{a}_j)(x)\big)^q w(x)\,d\mu(x)\right)^{1/q}\leq w(B_j)^{1/q -1/p}.
$$
Hence ${\rm supp}\,\mathcal{A}( \mathfrak{a}_j)\subseteq B_j$ and  one may apply H\"older inequality to obtain that for $p\in (0,1]$ and $q\in (1,\infty)$,
\begin{eqnarray*}
	\int_X \big(\mathcal{A}(\mathfrak{a}_j)(x)\big)^p w(x)\,d\mu(x) &\leq & \left(\int_{B_j}   \big(\mathcal{A}(\mathfrak{a}_j)(x)\big)^q w(x)\,d\mu(x)\right)^{p/q}   \left(\int_{B_j} w(x)\,d\mu(x)\right)^{1-p/q}\\
	&\leq& w(B_j)^{p(1/q -1/p)}w(B_j)^{1-p/q} =1.
\end{eqnarray*}
This, combined with the  fact that  $(\sum_k |d_k|)^p\leq \sum_k |d_k|^p$ for $p\in (0,1]$, deduces
$$
    \int_X \big(\mathcal{A}(F)(x)\big)^p w(x)\,d\mu(x)\leq \sum_j |\lambda_j|^p \int_X \big(\mathcal{A}(\mathfrak{a}_j)(x)\big)^pw(x)\,d\mu(x)\leq \sum_j |\lambda_j|^p ,
$$
as desired.

\medskip

Next  we will  prove  (i).  In the sequel we fix  parameters $\eta=1/2$ and $\gamma\in (0,1)$ such that the estimate \eqref{eqn:tent} in Lemma~\ref{tent space} holds. For each $k\in \mathbb{Z}$,
\begin{equation}\label{eqn:Omegak}
	  \Omega_k:=\left\{x\in X: \mathcal{A}(F)(x)>2^k/C_{\gamma,1/2,n}\right\},
\end{equation}
where $C_{\gamma, 1/2,n}$ is the constant in \eqref{eqn:tent} for  $\eta=1/2$. Then
$$
    \Omega^*_k= \{x\in X: {\mathcal M}(\chi_{\Omega_k})(x)> 1-\gamma  \}
$$
by \eqref{eqn:O*}. Clearly $\Omega_k\subseteq \Omega_k^{\ast}$ and $\mu(\Omega_k^{\ast})\leq  C_\gamma\, \mu(\Omega_k)$ for every $k\in{\Bbb Z}$. Moreover,
\begin{equation}\label{eqn:Omegakc}
    (\Omega^*_k)^c=(^c \Omega_k)^* \quad \text{and}\qquad (T_{\alpha} (\Omega^*_k))^c=\mathcal{R}_\alpha((^c \Omega_k)^*),  \ \ \  \text{for}\  \alpha>0.
\end{equation}

For every open set $O\subseteq X$, denote
$$
    \widehat{O}:=T_{1-\eta}(O)=T_{1/2}(O)
$$
for simplicity. It is known from \cite[P. 130]{Ru06} that
\begin{align}\label{support of F}
    {\rm supp} \, F \subseteq \bigcup_k \, T_{1/2}(\Omega_k^*).
\end{align}

Now let $\{Q_j^k\}_{j\in \mathbb{Z}}$ be a  Whitney decomposition of  $\Omega^*_k$ from Lemma~\ref{lem:Whitney}. Let  $C_1>0$ which will be determined later, and  $B_j^k:=B(x_j^k,C_1{\rm diam\,}(Q_j^k))$, where $x_j^k$ is the   center point of $Q_j^k$.  Then for every $j,k\in{\Bbb Z}$, we define
 \begin{equation}\label{e3.2}
    \Delta_j^k:=T(B_j^k)\cap\big(Q^k_j\times (0,+\infty)\big)\cap \big(\widehat{\Omega^*_k} \big\backslash {\widehat{\Omega^*_{k+\!1}}}\big).
\end{equation}
From this definition, one may combine   Lemma~\ref{lem:Whitney}  to see $\Delta_j^k$'s are disjoint for different $j$ or $k$. Furthermore, notice that $\widehat{\Omega_k^*}\subseteq \widehat{\Omega_{k'}^*}$ for any $k'<k$. Then for any given $(y,t)\in X\times (0,+\infty)$, we can set
$$
     k_0:=\max_{k\in \mathbb{Z}:\, (y,t)\in  \widehat{\Omega_k^*}} k,
$$
and obtain $(y,t)\in \widehat{\Omega_{k_0}^*}\setminus \widehat{\Omega_{k_0+1}^*}$, which yields $d(y, {^c \Omega_{k_0}^*})\geq (1-1/2)t=t/2$. Moreover, there exists a unique $Q_{j_0}^{k_0}$ from the Whitney decomposition of $\Omega_{k_0}^*$ such that $y\in Q_{j_0}^{k_0}$. Since $d(Q_{j_0}^{k_0}, {^c \Omega_{k_0}^*})\leq \delta^{-2} {\rm diam }(Q_{j_0}^{k_0})$, then $t\leq 2\,d(y,{^c \Omega_{k_0}^*})\leq   2\big(\delta^{-2}+1\big) {\rm diam }(Q_{j_0}^{k_0}) $ and so $(y,t)\in T(B_{j_0}^{k_0})$ by taking $C_1$ large sufficiently,  for instance, one may take
\begin{equation}\label{eqn:C1-1}
	  C_1\geq 2\delta^{-2}+3.
\end{equation}

From this and (\ref{support of F}),  we can write
\begin{eqnarray*}
       F(x,t)&=&\sum_k \sum_j F(x,t)\chi_{\Delta_j^k} (x,t)\\
       &=&\sum_k \sum_j \lambda_j^k \bigg( \frac{1}{\lambda_j^k} F(x,t) \chi_{\Delta_j^k}(x,t)\bigg)
       =: \sum_k \sum_j \lambda_j^k \mathfrak{a}_j^k(x,t),
\end{eqnarray*}
where
$$
	     \lambda_j^k =2^k w(B_j^k)^{1/p}.
$$

Let us check that $\{\mathfrak{a}_j^k\}$ are $q$-atoms of $T_{2,w}^p(X)$. Note that  $q_1$-atoms of $T_{2,w}^p(X)$ must be $q_2$-atoms of $T_{2,w}^p(X)$ for any $q_1>q_2$. Hence, we only need to check that $\{\mathfrak{a}_j^k\}$ are $q$-atoms of $T_{2,w}^p(X)$ for $q$ large sufficiently.  It is easy to see that  ${\rm supp}\, \mathfrak{a}_j^k \subseteq T({B_j^k})$. We will verify the size condition (ii) of Definition~\ref{def:tentatom} for $q$ large sufficiently.

For $w\in A_\infty$, there exists some $q_0\geq 1$ such that $w\in A_{q_0}$ and then $w\in A_q$ for any $q\geq q_0$. It follows from the dual result (Lemma~\ref{lem:tentdual}) that when $q\geq q_0$,
\begin{eqnarray*}
	\|\mathfrak{a}_j^k\|_{T_{2,w}^q(X)} =\sup_{\|b_j^k\|_{T_{2,w^{-1/(q-1)}}^{q'}}\leq 1} \left|  \iint_{\Delta_j^k} \mathfrak{a}_j^k(y,t)  b_j^k(y,t)\,d\mu(y)\frac{dt}{t}  \right|.
\end{eqnarray*}
From the definition of $\Delta_j^k$ in \eqref{e3.2} and the relationship \eqref{eqn:Omegakc}, one may apply Lemma~\ref{tent space} to obtain
 \begin{eqnarray*}
 	& &\left|\iint_{\Delta_j^k} \mathfrak{a}_j^k(y,t)  b_j^k(y,t)\,d\mu(y)\frac{dt}{t}\right| \\
 	&\leq& C_{\gamma, 1/2,n}\int_{\Omega_{k+1}^c} \iint_{\Gamma(x)} |\mathfrak{a}_j^k(y,t)| \, |b_j^k(y,t)|\, \frac{d\mu(y)}{V(x,t)}\frac{dt}{t} \,d\mu(x)\\
 	&=& C_{\gamma, 1/2,n}\int_{\Omega_{k+1}^c\cap B_j^k} \iint_{\Gamma(x)} |\mathfrak{a}_j^k(y,t)| \, |b_j^k(y,t)|\, \frac{d\mu(y)}{V(x,t)}\frac{dt}{t} \,d\mu(x)\\
 	&\leq& C_{\gamma, 1/2,n}\int_{\Omega_{k+1}^c\cap B_j^k} \left(\iint_{\Gamma(x)} |\mathfrak{a}_j^k(y,t)|^2\, \frac{d\mu(y)}{V(x,t)}\frac{dt}{t} \right)^{1/2} \left(\iint_{\Gamma(x)} |b_j^k(y,t)|^2\, \frac{d\mu(y)}{V(x,t)}\frac{dt}{t} \right)^{1/2} d\mu(x).
 \end{eqnarray*}
By the definition of $\Omega_{k+1}$ in  \eqref{eqn:Omegak}, we have
$$
     \left(\iint_{\Gamma(x)} |\mathfrak{a}_j^k(y,t)|^2\, \frac{d\mu(y)}{V(x,t)}\frac{dt}{t} \right)^{1/2}\leq \frac{2^k}{ C_{\gamma, 1/2,n} \,\lambda_j^k},\  \    \   {\rm for\  all } \  \, x\in \Omega_{k+1}^c.
$$
Then
\begin{eqnarray*}
	\left|\iint \mathfrak{a}_j^k(y,t)  b_j^k(y,t)\,d\mu(y)\frac{dt}{t}\right|  &\leq& \frac{2^k}{\lambda_j^k} \int_{B_j^k} \left(\iint_{\Gamma(x)} |b_j^k(y,t)|^2\, \frac{d\mu(y)}{V(x,t)}\frac{dt}{t} \right)^{1/2} d\mu(x)\\
	&=&\frac{2^k}{\lambda_j^k} \int_{B_j^k} w(x)^{1/q} \mathcal{A}(b_j^k)(x) w(x)^{-1/q}\,d\mu(x)\\
	&\leq & \frac{2^k}{\lambda_j^k} w(B_j^k)^{1/q} \left(\int \big(\mathcal{A}(b_j^k(x))\big)^{q'} w^{-1/(q-1)}  \, d\mu(x)\right)^{1/{q'}}\\
	&\leq& \frac{2^k}{\lambda_j^k}w(B_j^k)^{1/q} \leq w(B_j^k)^{1/q-1/p},
\end{eqnarray*}
as desired.  That is, every $\mathfrak{a}_j^k$ is a q-atom of  $T_{2,w}^p(X)$.

Furthermore, noticing that for every $j,k\in \mathbb{Z}$,
$$
     |\lambda_j^k|^p=2^{kp}w(B_j^k)\leq C 2^{kp} w(Q_j^k)
$$
follows from Lemma~\ref{lee2.1} and the construction of $B_j^k$, where  $C$ depends on $w$ and the parameter $C_1$ in \eqref{eqn:C1-1}. Therefore,
\begin{eqnarray*}
	\sum_k \sum_j |\lambda_j^k|^p&\leq & C  \sum_k \sum_j 2^{kp}w(Q_j^k)\leq C \sum_k 2^{kp} w(\Omega_k^*).
\end{eqnarray*}
Recall that $w\in A_\infty$ implies that   $w\in A_{q_0}$ for some $q_0>1$.  By using the fact that $\mathcal{M}$ is bounded on $L_w^{q_0}(X)$, then we  have
\begin{align}
w(\Omega_k^*)=w(\{x\in X: {\mathcal M}(\chi_{\Omega_k})(x)> 1-\gamma  \})\leq C(q_0,\gamma) w(\Omega_k).
\end{align}

Therefore, we can obtain
\begin{eqnarray*}
	\sum_k \sum_j |\lambda_j^k|^p &\leq& C\sum_k 2^{kp} w\left(\left\{x: \mathcal{A}(F)(x)>2^k/C_{\gamma,1/2,n}\right\}\right)\\
	&\leq& C p\sum_{k} \int_{2^{k-1}/C_{\gamma,1/2,n}}^{2^k/C_{\gamma,1/2,n}}   \lambda^{p-1} w\left(\left\{x: \mathcal{A}(F)(x)>\lambda\right\}\right)d\lambda\\
	&\leq& C\|F\|_{T_{2,w}^p(X)}^p.
\end{eqnarray*}
We complete the proof of Theorem~\ref{thm:tentatom}.
\hfill $\Box$

\smallskip

\begin{remark}\label{rem:convergence}
Let $0<p\leq 1$ and $w\in A_\infty$. In the part (i) of Theorem~\ref{thm:tentatom}, we have shown the for any $F\in T_{2,w}^p(X)$, the atomic decomposition  $F=\sum_{j,k\in \mathbb{Z}} \lambda_j^k \mathfrak{a}_j^k$ holds in   almost everywhere.  One may apply Lebesgue's dominated convergence theorem  to see $F=\sum_{j,k\in \mathbb{Z}} \lambda_j^k \mathfrak{a}_j^k$ in $T_{2,w}^p(X)$. Moreover, if $F\in T_{2,w}^p(X)\cap T_2^2(X)$, where the unweighted tent space $T_2^2(X)$ is introduced in Definition~\ref{def:unweightedTent}. We can apply the argument in Proposition~4.10 in \cite{HLMMY} to obtain $F=\sum_{j,k\in \mathbb{Z}} \lambda_j^k \mathfrak{a}_j^k$ converges also in $T_2^2(X)$. Summarily, if $F\in T_{2,w}^p(X)\cap T_2^2(X)$, $0<p\leq 1$ and $w\in A_\infty$, then the atomic decomposition $F=\sum_{j,k\in \mathbb{Z}} \lambda_j^k \mathfrak{a}_j^k$, established in (i) of Theorem~\ref{thm:tentatom}, holds in $T_{2,w}^p(X)$, $T_2^2(X)$ and almost everywhere.
\end{remark}

\bigskip

\section{ Applications: weighted Hardy spaces associated to nonnegative  self-adjoint operators and atomic characterization }
\setcounter{equation}{0}

In this section,  we will apply Theorem \ref{thm:tentatom} to show  an atomic decomposition of weighted Hardy spaces $H_{L,w}^p(X)$ associated to $L$, $0<p\leq 1$.

\subsection{Assumption (H)}

The following will be assumed throughout this section:  assume that operator $L$ is a non-negative self-adjoint operator on $L^2(X,d\mu)$ and that the semigroup $e^{-tL}$, generated by $-L$ on $L^2(X,d\mu)$, has the kernel $p_{t}(x,y)$ which  satisfies the following  Gaussian upper bound. That is, there exist constants $C, c>0$ such that
$$
    |p_{t}(x,y)| \leq \frac{C}{V(x,\sqrt{t})} \exp\Big(-{d(x,y)^2\over c\,t}\Big)
    \leqno{(GE)}
$$
for all $t>0$ and $x,\, y\in X$.

We note that such estimates are typical for elliptic or sub-elliptic differential operators of second order (see for instance, \cite{DOS}).

\smallskip

\subsection{Finite speed propagation for the wave equation}
Let $L$ be an operator satisfying the assumption ${\bf (H)}$, and $E_L(\lambda)$ denotes its spectral decomposition, then for every bounded Borel function $F:[0,\infty)\to{\Bbb C}$, one defines the operator $F(L): L^2(X)\to L^2(X)$ by the formula
\begin{equation}\label{e2.1}
    F(L):=\int_0^{\infty}F(\lambda)\,dE_L(\lambda).
\end{equation}
 In particular, the  operator $ \cos(t\sqrt{L})$  is then well-defined  on $L^2(X)$. Moreover, it follows from Theorem 3 of \cite{CS} that there exists a constant $c_0$ such that  the Schwartz kernel $K_{\cos(t\sqrt{L})}(x,y)$ of $\cos(t\sqrt{L})$ satisfies
\begin{equation}\label{e2.2}
     {\rm supp} K_{\cos(t\sqrt{L})}(x,y)\subseteq  \big\{(x,y)\in X\times X: \, d(x,y)\leq   c_0t\big\}.
 \end{equation}
 By the Fourier inversion formula, whenever $F$ is an even bounded Borel function with $\hat{F} \in L^1(\mathbb{R})$, we can  write $F(\sqrt{L})$ in terms of $\cos(t\sqrt{L})$. Concretely, by recalling (\ref{e2.1}) we have
$$
    F(\sqrt{L})=(2\pi)^{-1}\int_{-\infty}^{\infty}{\hat F}(t)\cos(t\sqrt{L})\,dt,
$$
which, when combined with (\ref{e2.2}), gives
\begin{equation}\label{e2.3}
    K_{F(\sqrt{L})}(x,y)=(2\pi)^{-1}\int_{|t|\geq  c_0^{-1}d(x,y)}{\hat F}(t) K_{\cos(t\sqrt{L})}(x,y)\,dt.
\end{equation}

\begin{lemma}\label{le2.5}
Let $\varphi\in C^{\infty}_0(\mathbb R)$ be even, ${\rm supp}\,\varphi \subseteq [-c_0^{-1}, c_0^{-1}]$. Let $\Phi$ denote the Fourier transform of $\varphi$. Then for each $k=0,1,\cdots$, and for every $t>0$, the kernel
$K_{(t^2L)^{k}\Phi(t\sqrt{L})}(x,y)$ of $(t^2L)^{k}\Phi(t\sqrt{L})$ satisfies
\begin{equation}\label{e2.4}
\hspace{0.5cm}
    {\rm supp}\ \!  K_{(t^2L)^{k}\Phi(t\sqrt{L})}  \subseteq \big\{(x,y)\in X\times X: \, d(x,y)\leq t\big\}
\end{equation}
 and
\begin{equation}\label{e2.5}
    |K_{(t^2L)^{k}\Phi(t\sqrt{L})}(x,y)| \leq \frac{C_k}{V(x,t)},
\end{equation}
for all $t>0$ and $x,y\in X$.
\end{lemma}

\begin{proof}
We refer the reader to  Lemma~3.5 of \cite{HLMMY}. See also  Lemma~2.1 of \cite{GY}.
\end{proof}

\begin{lemma}\label{le2.6}
   For every $k=0,1,\cdots$, there exist two positive constants $C_k, c_k$ such that the kernel $p_{t, k}(x,y)$ of the operator $ (t^2L)^{k} e^{-t^2{L}}$ satisfies
\begin{equation}\label{pt}
    |p_{t, k}(x,y)|\leq    \frac{C_k}{V(x,t)}\exp\Big(-{d(x,y)^2\over c_kt^2}\Big)
\end{equation}

\noindent for all $t>0$ and almost every $x,y\in X$.
\end{lemma}

\begin{proof} For the proof, we refer the reader to
\cite[Theorem 6.17]{O}.
\end{proof}

\subsection {Spectral theory and Littlewood-Paley function associated to $L$}
For $s>0$, we define
$$
    {\Bbb F}(s):=\Big\{\psi:{\Bbb C}\to{\Bbb C}\ {\rm measurable}: \ \  |\psi(z)|\leq C {|z|^s\over ({1+|z|^{2s}})}\Big\}.
$$
Then for any non-zero function $\psi\in {\Bbb F}(s)$, we have that $\{\int_0^{\infty}|{\psi}(t)|^2\frac{dt}{t}\}^{1/2}<\infty$. Denote  $\psi_t(z):=\psi(tz)$ for $t>0$. It follows from the spectral theory in \cite{Y} that for any $f\in L^2(X)$,
\begin{eqnarray}\label{e3.13}\label{e3.13}
    \Big\{\int_0^{\infty}\|\psi(t\sqrt{L})f\|_{L^2(X)}^2\,{dt\over t}\Big\}^{1/2} &=&\Big\{\int_0^{\infty}\big\langle\,\overline{\psi}(t\sqrt{L})\, \psi(t\sqrt{L})f, f\big\rangle\, {dt\over t}\Big\}^{1/2}\nonumber\\
    &=&\Big\{\big\langle \int_0^{\infty}|\psi|^2(t\sqrt{L}) \,{dt\over t}f, f\big\rangle\Big\}^{1/2}\nonumber\\
    &\leq& \kappa \|f\|_{L^2(X)},
\end{eqnarray}
where $\kappa=\big\{\int_0^{\infty}|{\psi}(t)|^2\, {dt/t}\big\}^{1/2}$. The estimate  will be often used in this article.

Now we give the definition of Littlewood-Paley area function associated  to $L$. Given a function $f\in L^2(X)$, consider the square function associated to  the heat semigroup generated by the  operator $L$,
$$
    S_{L}(f)(x):=\Big(\iint_{  \Gamma(x)}\big|t^2Le^{-t^2{L}}f(y)\big|^2\,\frac{d\mu(y)}{V(x,t)}\frac{dt}{t}\Big)^{1/2}, \  \ \ \ \ x\in X.
$$
Similarly, we define $g^*_{\nu,\Psi}$  function associated to $L$.
\begin{equation}\label{s}
    g^*_{\nu,\Psi}(f)(x):=\bigg(\iint_{X\times (0,+\infty)}\Big(\frac{t}{t+d(x,y)}\Big)^{n\nu}\big|\Psi(t\sqrt{L})f(y)\big|^2\  \frac{d\mu(y)}{V(x,t)}\frac{dt}{t} \bigg)^{1/2},~~~\nu>1,
\end{equation}
where $\varphi$ and $\Phi$ are as in Lemma \ref{le2.5}, $\Psi(x):=x^{2\alpha}\Phi^3(x)$, $x\in \mathbb{R}$, $\alpha\geq n+1$ and $n$ is the parameter in \eqref{eqn:doubling2}.

\begin{lemma}\label{sfunction}
Assume $\nu>3$. Then for all $w\in A_s,\ 1<s<\infty$, there exists a constant $C=C(s,n,w)>0$  such that for  the following estimate holds:
$$
    \big\|g^*_{\nu,\Psi}(f)\big\|_{L^s_w({X})}+\big\|S_L(f)\big\|_{L^s_w({X})}\leq C\big\|f\big\|_{L^s_w({X})}.
$$
\end{lemma}

\begin{proof}
When $X=\mathbb{R}^n$, this lemma has proved in Lemma~5.1 of \cite{GY}, based on  a dyadic  Whitney decomposition on $\mathbb{R}^n$ and the functional calculi theory for operators. Due to estimates in Subsection~4.2 and Lemma~\ref{lem:Whitney},  the proof of Lemma~\ref{sfunction} is very similar to that of Euclidean case,  and we skip it.
\end{proof}

\vskip 0.1cm

\subsection{Weighted Hardy spaces  and weighted atoms}

\begin{definition}\label{def:Hp by area function}
 Suppose $w\in A_{\infty}$ and $0<p\leq 1$. We may define the spaces $H_{L,w}^p(X)$ as the completion of $\{f \in L^2(X): \|S_{L}(f)\|_{L_w^p(X)}< \infty\}$ with respect to $L^p_w(X)$-norm of the square function,
$$
    \big\|f\big\|_{H_{L,w}^p(X)}:=\big\|S_L(f)\big\|_{L^p_w(X)}.
$$
\end{definition}

\noindent {\bf Note:}
 (1)\, Let $F(y,t)=t^2Le^{-t^2L}f(y)$, then $S_L(f)(x)=\mathcal{A}(F)(x)$ for every $x\in X$ and  $\|f\|_{H_{L,w}^p(X)}=\|F\|_{T_{2,w}^p(X)}$.

 (2)\, Denote the Laplacian by $\triangle=-\sum\limits_{i=1}^n\partial^2_{x_i}$ and  $e^{-t {\triangle}}$  the heat semigroup on ${\mathbb R}^n$, given by
$$
    e^{-t\Delta}f(x)=\int_{{\mathbb R}^n}p_t(x-y)f(y)dy,\ \ t>0,
$$
where
$$
    p_t(x-y)={1\over (4\pi t)^{n/2}}\exp\Big(-{|x-y|^2\over 4t}\Big).
$$
 It is known that the  classical weighted  space $H^1_w({\mathbb R}^n)$ coincides with the space $H^1_{\triangle, w}({\mathbb R}^n)$  and their norms are equivalent. See \cite{G,ST}.

(3)\, When  $w= 1$,  the theory of Hardy spaces associated to operators was studied in \cite{ADM,AMR,DY2,HLMMY,HM,JY}.

\medskip

We next define  a $(p,q,M,w)$-atom  associated to the operator $L$.

\begin{definition}\label{da}
 Suppose that $M\in \mathbb{N}$, $w\in A_s ,1\leq s<\infty$ and $0<p\leq 1$. A function $a(x)\in L^2(X)$ is called a $(p,q,M,w)$-atom associated to  an operator $L$,   $1<q<\infty$, if there exist a function $b\in {\mathcal D}(L^M)$, the domain of $L^M$,  and a ball $B\subseteq X$ such that

\begin{enumerate}[\rm (i)]
	\item \ $a=L^M b$;
	
	\medskip
	
	\item  \ {\rm supp}\  $L^{k}b\subseteq B, \ k=0, 1,\cdots, M$;

    \medskip

   \item \ $\big\|(r_B^2L)^{k}b\big\|_{L_w^q(X)}\leq r_B^{2M} w(B)^{1/q-1/p},\ k=0,1, \cdots, M$.
\end{enumerate}
\end{definition}

\begin{remark}\label{re3.2}
When  $w=1$ and $p=1$,   a theory of Hardy space $H^1_{L}(X)$ associated to $L$  was  presented  in \cite{HLMMY}, including an atomic (or molecular) decomposition, square function characterization and duality of Hardy and BMO spaces.
\end{remark}

\begin{remark}\label{atom q1>q2}
It follows directly from H\"older's inequality that a $(p,q_1,M,w)$-atom is also a $(p,q_2,M,w)$-atom whenever $q_1>q_2$.
\end{remark}

\smallskip

\begin{definition} \label{def:Hatom}
Let $M$, $w$, $p$ and $q$ be the same as above. The weighted atomic Hardy space ${H}^{p,q,M}_{L,w}(X)$ is defined as follows. We will say that $f= \sum\lambda_j a_j$ is an atomic $(p,q,M,w)$-representation (of $f$) if $
\{\lambda_j\}_{j=0}^{\infty}\in {\ell}^p$, each $a_j$ is a $(p,q,M,w)$-atom, and the sum converges in $L^2(X).$  Set
$$
    \mathbb {H}^{p,q,M}_{L,w}(X):=\Big\{f:  f \mbox{  has an atomic $(p,q,M,w)$-representation} \Big\},
$$
with the norm $\big\|f\big\|_{\mathbb {H}^{p,q,M}_{L,w}(X)}$ given by
$$
    \inf\Big\{\Big(\sum_{j=0}^{\infty}|\lambda_j|^p\Big)^{1/p}: \  \  f=\sum\limits_{j=0}^{\infty}\lambda_ja_j \ \mbox{ is an atomic $(p,q,M,w)$-representation}\Big\}.
$$

\smallskip

The space ${H}^{p,q,M}_{L,w}(X)$ is then defined as the completion of $\mathbb {H}^{p,q,M}_{L,w}(X)$ with respect to this norm.
 \end{definition}

\subsection{Atomic characterization of weighted Hardy spaces }
The aim of this part is to prove Theorem \ref{thm:Hardyatom}, the second main result of this paper. By Definitions \ref{def:Hp by area function} and \ref{def:Hatom},  Theorem  \ref{thm:Hardyatom} will be a direct corollary if we have  proved the following Theorem \ref{th1.2}.

\begin{theorem}\label{th1.2}
Suppose that  $w\in A_s, 1<s<\infty$ and $0<p\leq 1<q<\infty$.

{\rm (i)} Let $M\in {\mathbb N}$. If $f\in H_{L,w}^p(X)\cap L^2(X)$, then there exist a family of $(p,q,M,w)$-atoms $\{ a_i\}_{i=0}^{\infty}$ and a sequence of numbers of $\{\lambda_i\}_{i=0}^{\infty}$ such that $f$ can be represented in the form $f=\sum_{i=0}^{\infty} \lambda_i a_i$, and the sum converges in the sense of $L^2(X)$-norm. Moreover,
$$
    \Big(\sum_{i=0}^{\infty}|\lambda_i|^p\Big)^{1/p}\leq C\big\|f\big\|_{H_{L,w}^p(X)}.
$$

{\rm (ii)} Suppose that $M\in {\mathbb N}$, $\displaystyle   M>\frac{(s-p)n}{2p}$ and $q\geq s$. Let $f=\sum_{i=0}^{\infty} \lambda_i a_i$, where $\{\lambda_i\}_{i=0}^{\infty}\in \ell^p$, the $a_i$ are  $(p,q,M,w)$-atoms, and  the sum converges in $L^2(X)$. Then $f\in H_{L,w}^p(X)\cap L^2(X)$ and

$$
    \Big\|\sum_{i=0}^{\infty} \lambda_i a_i\Big\|_{H_{L,w}^p(X)}\leq C\Big(\sum_{i=0}^{\infty} |\lambda_i|^p\Big)^{1/p}.
$$
\end{theorem}

\smallskip

To prove the part (i) of  Theorem~\ref{th1.2}, we will study the relationship between weighted tent spaces $T_{2,w}^p(X)$ and weighted Hardy spaces $H_{L,w}^p(X)$.

Let $\varphi$ and $\Phi$ be the same as in Lemma \ref{le2.5}, $\Psi(x):=x^{2\alpha}\Phi^3(x)$, $\alpha\geq n+1$, $M\in \mathbb{N}$ and $n$ is the parameter in \eqref{eqn:doubling2}. Consider the operator $\pi_{\Psi,L}$ initially defined on  the dense subspace of  $T_{2,w}^p(X) \ \   (0<p<\infty)$ of $F$ with compact support in $X\times (0,+\infty)$ by
\begin{equation}\label{e9.7}
\pi_{\Psi,L}(F)(x):= \int_0^{\infty}\Psi(t\sqrt L)\big(F(\cdot,\, t)\big)(x)\,{dt\over t}.
\end{equation}

\begin{prop}\label{prop:TentHardy}
Suppose  $w\in A_q$, where $q\in (1,\infty)$.
\begin{itemize}
\item [\rm(i)] $\pi_{\Psi,L}$ is bounded from $T_{2,w}^q(X)$ to $L_w^q(X)$.

\medskip

\item [\rm(ii)] Let $0<p\leq 1$ and $\alpha=M+n+1$. $\pi_{\Psi,L}$ maps every $q-$atom of $T_{2,w}^p(X)$, $\mathfrak{a}$, to a $(p,q,M,w)$-atom, up to multiplication by a harmless constant independent of $\mathfrak{a}$,
$$
    \|\pi_{\Psi,L}(\mathfrak{a})\|_{L_w^q(X)}\leq C \|\mathfrak{a}\|_{T_{2,w}^q(X)}.
$$
\end{itemize}
\end{prop}

\noindent{\it Proof}.\
{\rm (i)}. We will use duality argument. For every $g\in L_{w^{-1/(q-1)}}^{q'}$ with $\|g\|_{L_{w^{-1/(q-1)}}^{q'}} \leq 1$, one can obtain
\begin{eqnarray*}
    \left|\int_X \pi_{\Psi,L}(F)(x)g(x)\,d\mu(x)\right| &=&\left|\iint_{X\times (0,+\infty)} F(y,t)\Psi(t\sqrt{L})g(y) \,d\mu(y)\frac{dt}{t}\right|\\
      &\leq& C_n \|F\|_{T_{2,w}^q(X)} \|\Psi(t\sqrt{L})g\|_{T_{2,w^{-1/(q-1)}}^{q'}(X)}\\
      &\leq& C_n \|F\|_{T_{2,w}^q(X)} \|g^*_{\nu,\Psi}(g)\|_{L_{w^{-1/(q-1)}}^{q'}(X)}\\
     &\leq& C_{n,q,w} \|F\|_{T_{2,w}^q(X)} \|g\|_{L_{w^{-1/(q-1)}}^{q'}(X)}\\
     &\leq & C_{n,q,w} \|F\|_{T_{2,w}^q(X)},
\end{eqnarray*}
where  in the last two  inequality above we have used Lemma~\ref{sfunction} and $w\in A_q$.

\smallskip

{\rm (ii).}  Let  $\mathfrak{a}$ be any  $q-$atom of $T_{2,w}^p(X)$. Then there exists a ball $B\subseteq X$  such that    ${\rm supp}\, \mathfrak{a}\subseteq T(B)$.
 Recall $\Psi(x)=x^{2\alpha}\Phi^3(x)$ and  $\alpha=M+n+1$ fixed. We can rewrite
$$
    \pi_{\Psi,L}(\mathfrak{a})=L^M b,  \   \   \text{where} \  \   b=  \int_0^\infty t^{2\alpha}L^{n+1} \Phi^3(t\sqrt{L}) (\mathfrak{a}(\cdot,t))\, \frac{{\rm d}t}{t}.
$$
From Lemma \ref{le2.5},  the integral kernel $K_{(t^2L)^i\Phi^3(t\sqrt{L})}(x,y)$ of the operator $(t^2L)^i\Phi^3(t\sqrt{L})$ satisfies
$$
    {\rm supp} \  K_{(t^2L)^i\Phi^3(t\sqrt{L})}(x,y) \subseteq \big\{(x,y)\in   X \times   X: d(x,y)\leq t\big\}.
$$
This, together with the fact for any  $(y,t)\in {\rm supp}\, \mathfrak{a}$, we have $y\in B$ and $t\leq r_B$, where $r_B$ denotes the radius of $B$, implies that for every $i=0, 1,\cdots, M$,
$$
    {\rm supp}  \big(L^i b\big) \subseteq 3B.
$$
To continue, for every $i=0,1,\cdots, M$,   we apply (i) of Proposition~\ref{prop:TentHardy} to obtain
$$
    \pi_{\Psi,L,i}(F)(x):=\int_0^\infty  (t^2L)^{n+1+i}\Phi^3(t\sqrt{L})   (F(\cdot, t))(x) \,\frac{dt}{t}
$$
is bounded from $T_{2,w}^q(X)$ to $L_w^q(X)$. Notice that
$$
    (r_B^2 L)^i b=r^{2M}_{B}\int_0^\infty  (t^2L)^{n+1+i}
\Phi^3(t\sqrt{L}) \Big(\frac{t^{2(M-i)}}{r_B^{2(M-i)}}\mathfrak{a}(\cdot,t)\Big)\,\frac{dt}{t}=r_B^{2M}\pi_{\Psi,L,i}\Big(\frac{t^{2(M-i)}}{r_B^{2(M-i)}}\mathfrak{a}(\cdot,t)\Big),
$$
we have
\begin{eqnarray*}
    \| (r_B^2 L)^i b\|_{L_w^q(X)}&=&r_B^{2M}\Big\|\pi_{\Psi,L,i}\Big(\frac{t^{2(M-i)}}{r_B^{2(M-i)}}\mathfrak{a}(\cdot,t)\Big)\Big\|_{L_w^q(X)}\\
    &\leq& C r_B^{2M} \Big\|\frac{t^{2(M-i)}}{r_B^{2(M-i)}}\mathfrak{a}\Big\|_{T_{2,w}^q(X)}\leq C r_B^{2M} w(B)^{1/q-1/p},
\end{eqnarray*}
by using  the fact $t\leq r_B$.

From the above, $\pi_{\Psi,L}(\mathfrak{a})$ is a  $(p,q,M,w)$-atom up to a normalization by a  multiplicative constant, and we complete the proof.
\hfill $\Box$

\medskip

Now we turn to show (i)  of  Theorem~\ref{th1.2} by applying Proposition~\ref{prop:TentHardy}.

\medskip

\noindent{\it Proof of  (i) of Theorem~\ref{th1.2}.} \ For  $f\in H_{L,w}^p(X)\cap L^2(X)$, where $w\in A_s$ with $1< s<\infty$, by the spectral theory (see for example \cite{Y}),  we have the following ``Calder\'on reproducing formula",
\begin{equation}\label{eqn:Calderon}
    f(x)=c_{\Psi}\int_0^\infty \Psi(t\sqrt{L})t^2{L}e^{-t^2{L}}f(x)\,\frac{dt}{t},
\end{equation}
with the integral converging in $L^2(X)$, where $\Psi(x):=x^{2\alpha}\Phi^3(x)$, $\alpha=M+ n+1$, and $\Phi$ as that of Lemma \ref{le2.5}.
 Let $F(x,t):=t^2L e^{-t^2 L}f(x)$ for $(x,t)\in X\times (0,+\infty)$.  Clearly,  $F\in T_{2,w}^p(X)$ with
\begin{equation}\label{TentFaux}
      \|F\|_{T_{2,w}^p(X)}=\big\|f\big\|_{H_{L,w}^p(X)}.
\end{equation}
Hence, one may apply  Theorem~\ref{thm:tentatom} to obtain $F=\sum_{j,k\in \mathbb{Z}}\lambda_j^k \mathfrak{a}_j^k(x,t)$ almost  everywhere, where $\mathfrak{a}_j(x,t)$ are $q$-atoms of $T_{2,w}^p(X)$, for any $q\geq s$,  and
\begin{equation}\label{eqn:lambda}
    \Big( \sum_{j,k} |\lambda_j^k|^p\Big)^{1/p}\leq C \|F\|_{T_{2,w}^p(X)}=C\big\|f\big\|_{H_{L,w}^p(X)}.
\end{equation}
Thus, one can write
$$
    f(x)=\sum_{j,k\in \mathbb{Z}}\lambda_j^k  {a}_j^k(x),
$$
where
$$
    a_j^k:=c_\Psi \pi_{\Psi,L}(\mathfrak{a}^k_j)=c_\Psi \int_0^\infty \Psi(t\sqrt{L}) (\mathfrak{a}_j^k(\cdot, t))(x)\,\frac{dt}{t}.
$$
Noting $q\geq s$,  we can apply (ii) of Proposition~\ref{prop:TentHardy} to get that for every $j,k\in \mathbb{Z}$, $a_j^k$ is  a $(p,q,M,w)$-atom (up to multiplication by some harmless constant).  By Remark \ref{atom q1>q2}, $a_j^k$ is also  a $(p,r,M,w)$-atom, for any $1<r<s$.

It remains to  verify that  the sum $f(x)=\sum_{j,k\in{\Bbb Z}}\lambda_j^k a_j^k$ converges in $L^2(X)$. In fact, it follows from combining \eqref{e3.13}  that $F=t^2Le^{-t^2L}f\in T_2^2(X)$ since $f\in L^2(X)$. Then Remark~\ref{rem:convergence} tells that $F=\sum_{j,k\in \mathbb{Z}}\lambda_j^k \mathfrak{a}_j^k(x,t)$ converges in $T_2^2(X)$. By applying a special (unweighted) case of Proposition~\ref{prop:TentHardy}, $\pi_{\Psi,L}$ is bounded from $T_2^2(X)$ to $L^2(X)$.  Hence $f(x)=\sum_{j,k\in{\Bbb Z}} \lambda_j^k a_j^k$ converges in $L^2(X)$  by applying Lemma~4.12 in \cite{HLMMY}.
\hfill $\Box$

\vskip 0.2cm

Let us  state a lemma  before we  prove the part (ii) of Theorem~\ref{th1.2}.   .
\begin{lemma}\label{le4.1}
Fix $M\in {\Bbb N}$, $w\in A_\infty$ and $0<p\leq 1$. Assume that $T$ is a non-negative sublinear  operator, satisfying the weak-type (2,2) bound
\begin{equation}\label{e4.weak}
    \mu\big(\big\{x\in X: |Tf(x)|>\lambda\big\} \big)\leq \,C_T\, \, \lambda^{-2}
    \|f\|_{L^2(X)}^2,\quad \mbox{ for all}\  \lambda>0,
\end{equation}
and that for every $(p,q,M,w)$-atom $a$, we have
\begin{equation}\label{e4.1a}
    \|Ta\|_{L^p_w(X)}\leq C
\end{equation}
\noindent with constant $C$ independent of $a$. Then  $T$ is bounded from $\mathbb{H}_{L,w}^{p,q,M}(X)$ to $L^p_w(X)$, and
$$
    \big\|Tf\big\|_{L^p_w(X)}\leq C \big\|f\big\|_{\mathbb{H}_{L,w}^{p,q,M}(X)}.
$$
Consequently, by density, $T$ extends to a bounded operator from $H_{L,w}^{p,q,M}(X)$ to $L_w^p(X)$.
\end{lemma}

\begin{proof}
The proof of this lemma is similar to that of Lemma 4.3 in \cite{HLMMY}, where $w(x)=1$ and $p=1$. We give its proof for completeness.

For $f\in    \mathbb{H}_{L,w}^{p,q,M}(X)$, where $f=\sum_{j=0}^\infty \lambda_j a_j$ is an atomic $(p,q,M,w)$-representation such that
$$
      \|f\|_{\mathbb{H}_{L,w}^{p,q,M}(X)} \approx \Big(\sum_{j=0}^\infty |\lambda_j|^p\Big)^{1/p}.
$$
Since the sum above converges in $L^2(X)$ from Definition~\ref{def:Hatom}, then $\displaystyle \underset{N\to \infty}{\lim\sup}\,\|f^N\|_{L^2(X)} =0$, where $f^N=\sum_{j>N} \lambda_j a_j$ with $N\in \mathbb{N}$. Moreover, for a non-negative sublinear operator $T$, we know for every $\varepsilon>0$,
\begin{eqnarray*}
	& &\mu\Big\{x\in X:\  \   |Tf(x)|-\sum_{j=0}^\infty |\lambda_j|\,|Ta_j(x)| >\varepsilon \Big\} \\
	&\leq& \underset{N\to \infty}{\lim\sup}\    \mu\Big\{x\in X:\  \   |Tf(x)|-\sum_{j=0}^N |\lambda_j|\,|Ta_j(x)| >\varepsilon \Big\}\\
	&\leq& \underset{N\to \infty}{\lim\sup}\    \mu\big\{x\in X:\  \  |T(f^N(x))|>\varepsilon \big\}\\
	&\leq& C \varepsilon^{-2} \,\underset{N\to \infty}{\lim\sup} \,\|f^N\|_2^2 =0,
\end{eqnarray*}
where the last inequality above follows from the weak-type $(2,2)$ of $T$. Therefore,
$$
    |Tf(x)|\leq \sum_{j=0}^\infty |\lambda_j|\, |Ta_j(x)|
$$
for almost every $x\in X$. This, together with \eqref{e4.1a}, deduces
\begin{eqnarray*}
	\|Tf\|_{L_w^p(X)}^p \leq \sum_{j=0}^\infty |\lambda_j|^p \, \|Ta_j\|_{L_w^p(X)}^p\leq C \sum_{j=0}^\infty |\lambda_j|^p  \approx \|f\|_{\mathbb{H}_{L,w}^{p,q,M}(X)} ^p
\end{eqnarray*}
for $0<p\leq 1$, as desired.
 \end{proof}

\vskip 0.1cm

\noindent{\it Proof of  (ii) of Theorem~\ref{th1.2}.} \ It is clear that $\mathbb{H}_{L,w}^{p,q,M}(X)\subseteq L^2(X)$.
From Lemma \ref{le4.1},  it remains to establish a uniform $L^p_w(X)$ bound on atoms. That is to say,  there exists a constant $C>0$ such that for every $(p,q,M,w)$-atom $a$ associated to a ball $B=B(x_B,r_B)$ ($x_B$ is the center of $B$, $r_B$ is the radius of $B$),
\begin{eqnarray}\label{ee4.4}
    \big\|S_L(a)\big\|_{L^p_w(X)}\leq C.
\end{eqnarray}
 One can write
\begin{eqnarray*}
    \int \big(S_L(a)(x)\big)^p w(x)\,d\mu(x)&=& \int_{2B} \big(S_L(a)(x)\big)^p w(x)\,d\mu(x) +
    \int_{(2B)^c}  \big(S_L(a)(x)\big)^p w(x)\,d\mu(x) \\
    &=:& {\rm I}_1+ {\rm I}_2.
\end{eqnarray*}

Since $w\in A_s$ and $1< s\leq q$, then $w\in A_q$.   It follows from  Lemma \ref{sfunction} and H\"older's inequality that
\begin{eqnarray*}
{\rm I}_1&\leq
&\Big(\int_{2B}\big(S_L(a)(x)\big)^qw(x)\, d\mu(x)\Big)^{p/q}\Big(\int_{2B}
w(x)\,d\mu(x)\Big)^{1-p/q}\\
&\leq& C\big\|a\big\|_{L^q_w(X)}^p w(2B)^{1-p/q}\\
&\leq& C w(B)^{p(1/q-1/p)}w(B)^{1-p/q}\leq C.
\end{eqnarray*}

\vskip 0.1cm

Let us consider ${\rm I}_2$.  Since $M>(s-p)n/(2p)$, we can choose a  number $N$ such that $2M>N>(s-p)n/p$.
For any $x\in (2B)^c$, one writes
\begin{eqnarray*}
S^2_L(a)(x)&=&\Big(\int_0^{r_B}+\int^{+\infty}_{r_B}\Big)
\int_{d(x,y)<t}\big|t^2Le^{-t^2L}a(y)\big|^2\,\frac{d\mu(y)}{V(x,t)}\frac{dt}{t}
=:\rm{J}_1+\rm{J}_2.
\end{eqnarray*}

\noindent We observe that $x\in (2B)^c,\  z\in B$ and $d(x,y)<t$
imply that $V(x,t)\approx V(y,t)$ and $r_B\leq d(x,z)<t+d(y,z)$, thus $d(x,x_B)\leq d(x,y)+d(y,z)+d(z,x_B)<3(t+d(y,z))$.  This, combined with the doubling volume property, deduces
$$
    \frac{1}{V(y,t)}\approx \frac{1}{V(x,t)} \leq C\left(\frac{t+d(y,z)}{t}\right)^n \frac{1}{V(x,3(t+d(y,z)))}\leq C\left(\frac{t+d(y,z)}{t}\right)^n \frac{1}{V(x,d(x,x_B))}.
$$

Hence one may apply  Lemma \ref{le2.6} to see
\begin{eqnarray*}
\rm{J}_1&\leq& C\int_0^{r_B}\!\!\int_{d(x,y)<t}\Big(\int
\frac{t^{N+n}}{V(y,t) \big(t+d(y,z)\big)^{N+n}}|a(z)|\,d\mu(z)
\Big)^2\,\frac{d\mu(y)}{V(x,t)}\frac{dt}{t}\\
&\leq& C\frac{1}{{V(x,d(x,x_B))}^2}\int_0^{r_B}\!\!\int_{d(x,y)<t}\Big(\int
\frac{t^{N}}{{d(x,x_B)}^{N}}|a(z)|\,d\mu(z)
\Big)^2\,\frac{d\mu(y)}{V(x,t)}\frac{dt}{t}\\
&\leq&
C\frac{\|a\|^2_{L^1(X)}}{V(x,d(x,x_B))^2d(x,x_B)^{2N}}\int_{0}^{r_B} t^{2N-1}\,dt\nonumber\\
&\leq&
C\frac{r_B^{2N}}{V(x,d(x,x_B))^2 d(x,x_B)^{2N}}\,\|a\|^2_{L^1(X)}.\nonumber
\end{eqnarray*}

\noindent Consider the  term ${\rm J}_2$. Notice that $a=L^Mb$, we apply Lemma \ref{le2.6} to obtain
\begin{eqnarray*}
{\rm J}_2&=&\int_{r_B}^\infty
\int_{d(x,y)<t}\big|{t^2}Le^{-t^2L}(L^Mb)(y)\big|^2\,
\frac{d\mu(y)}{V(x,t)}\frac{dt}{t}\\
&=& \int_{r_B}^\infty
\int_{d(x,y)<t}\big|(t^2L)^{M+1}e^{-t^2L}b(y)\big|^2\, \frac{d\mu(y)}{V(x,t)}\frac{dt}{t^{4M+1}}
\nonumber\\
&\leq& C\int_{r_B}^\infty\int_{d(x,y)<t}\Big(\int_B
\frac{t^{N+n}}{V(x,t)\big(t+d(y,z)\big)^{N+n}}|b(z)|\,d\mu(z)
\Big)^2\,\frac{d\mu(y)}{V(x,t)}\frac{dt}{t^{4M+1}}\nonumber\\
&\leq&
C\frac{\|b\|^2_{L^1(X)}}{V(x,d(x,x_B))^2d(x,x_B)^{2N}} \int_{r_B}^\infty \frac{dt}{ t^{4M-2N+1}}\,\nonumber\\
&\leq& C\frac{r_B^{2N-4M}}{V(x,d(x,x_B))^2d(x,x_B)^{2N}}\|b\|^2_{L^1(X)},\nonumber
\end{eqnarray*}
since $M>N/2.$

By the estimates of $\rm{J}_1$ and $\rm{J}_{2},$  we obtain
\begin{eqnarray*}
{\rm I}_2&\leq&
C\int_{(2B)^c}\frac{r_B^{Np}}{V(x,d(x,x_B))^p\,d(x,x_B)^{Np}}w(x)\,
d\mu(x)\,\Big(\|a\|_{L^1(X)}^p+r_B^{-2Mp}\|b\|_{L^1(X)}^p\Big).
\end{eqnarray*}

\noindent Using H\"older's inequality and the definition of $(p,q,M,w)$-atom, we get
\begin{eqnarray}\label{eee3.5}
\|a\|_{L^1(X)}+r_B^{-2M}\|b\|_{L^1(X)}&\leq& C w(B)^{1/q-1/p}\Big(\int_B w(x)^{{-q'}/{q}}\,d
\mu(x)\Big)^{{1}/{q'}}\nonumber\\
&\leq& C V(B)w(B)^{-1/p},
\end{eqnarray}
where in the last inequality above we have used the fact  $w\in A_s\subseteq A_q$ and the definition of $A_q$ weights. Moreover, it follows from (iv) of Lemma~\ref{lee2.1}  that,   $ w(2^{k+1}B)\leq C 2^{ns}  w(2^k B)$, $w(2^kB)\leq C2^{kns}  w(B)$ and
\begin{align}\label{ineq:weight}
    \bigg(\frac{V(B)}{V(2^k B)}\bigg)^p\leq [w]_{A_s}^{p/s} \bigg(\frac{w(B)}{w(2^k B)}\bigg)^{p/s}.
\end{align}
for every $k\geq 1$.
Therefore,
\begin{eqnarray}\label{eee3.6}
& &\int_{(2B)^c}\frac{r_B^{Np}}{V(x,d(x,x_B)) ^p\, d(x,x_B)^{Np}}w(x)\,d\mu(x) \nonumber\\
&\leq&
C\sum\limits_{k=1}^{\infty}\int_{2^{k+1}B\backslash 2^kB}\frac{r_B^{Np}}{V(2^k B)^p\,d(x,x_B)^{Np}}w(x)\,d\mu(x)\nonumber\\
&\leq&
C_{n,w,p,s}\sum\limits_{k=1}^{\infty} \frac{r_B^{Np}}{ (2^kr_B)^{Np}} \,  V(B)^{-p}\bigg(\frac{w(B)}{w(2^{k} B)}\bigg)^{p/s}w(2^{k+1}B) \\
&\leq&
C_{n,w,p,s}\sum\limits_{k=1}^{\infty} 2^{-kNp} V(B)^{-p} w(B)^{p/s} \big( 2^{kns }w(B) \big)^{1-p/s}    \nonumber\\
&\leq& C_{n,w,p,s} V(B)^{-p} w(B), \nonumber
\end{eqnarray}
where in the second inequality above we have used \eqref{ineq:weight} and in the last inequality above we have used $N>(s-p)n/p$. It follows from estimates (\ref{eee3.5}) and (\ref{eee3.6}) that ${\rm I}_2\leq C$,  which completes the proof of (\ref{ee4.4}) and the proof
of part (ii) of Theorem~\ref{th1.2}.
\hfill$\Box$

\bigskip

\noindent{\bf Acknowledgments.} We would like to thank Lixin Yan for helpful discussions.  L.~Song is supported by  NNSF of China (No.~11622113) and NSF for distinguished Young Scholar of Guangdong Province (No.~2016A030306040).

\end{document}